\documentclass[leqno,11pt]{article}

\usepackage[utf8]{inputenc}
\usepackage[T1]{fontenc}
\usepackage[english]{babel}

\usepackage[intlimits]{amsmath}
\usepackage{amsthm}
\usepackage{amsbsy}
\usepackage{amscd} 
\usepackage{amsfonts}
\usepackage{amsgen} 
\usepackage{amsmath}
\usepackage{amsopn} 
\usepackage{amssymb}
\usepackage{amstext}
\usepackage{amsthm} 
\usepackage{amsxtra}

\usepackage{MnSymbol}
\usepackage{mathbbol}
\usepackage{bbm}
\usepackage{mathrsfs}

\usepackage[all]{xy}

\usepackage{graphicx}
\usepackage{rotating}
\usepackage{graphics}
\usepackage{color}

\numberwithin{equation}{section}

\newtheorem{theoremcounter}{theoremcounter}[section]
\newtheorem{thmstarcounter}{thmstarcounter}

\newtheorem{corollary}[theoremcounter]{Corollary}
\newtheorem{lemma}[theoremcounter]{Lemma}

\newtheorem{proposition}[theoremcounter]{Proposition}
\newtheorem{theorem}[theoremcounter]{Theorem}
\newtheorem{thmstar}[thmstarcounter]{Theorem}

\theoremstyle{definition}

\newtheorem{definition}[theoremcounter]{Definition}

\newtheorem{example}[theoremcounter]{Example}
\newtheorem{remark}[theoremcounter]{Remark}

\newcommand{\cC}{\ensuremath{\mathcal{C}}}

\newcommand{\cN}{\ensuremath{\mathcal{N}}}

\newcommand{\cP}{\ensuremath{\mathcal{P}}}

\newcommand{\cU}{\ensuremath{\mathcal{U}}}
\newcommand{\cV}{\ensuremath{\mathcal{V}}}

\newcommand{\rE}{\ensuremath{\mathrm{E}}}

\newcommand{\rL}{\ensuremath{\mathrm{L}}}
\newcommand{\rM}{\ensuremath{\mathrm{M}}}

\newcommand{\rS}{\ensuremath{\mathrm{S}}}

\newcommand{\wh}{\widehat}

\newcommand{\amid}{\ensuremath{\, | \,}}

\newcommand{\eqstop}{\ensuremath{\, \text{.}}}
\newcommand{\eqcomma}{\ensuremath{\, \text{,}}}

\newcommand{\NN}{\ensuremath{\mathbb{N}}}
\newcommand{\ZZ}{\ensuremath{\mathbb{Z}}}

\newcommand{\CC}{\ensuremath{\mathbb{C}}}

\newcommand{\Hom}{\ensuremath{\mathop{\mathrm{Hom}}}}
\newcommand{\id}{\ensuremath{\mathrm{id}}}
\newcommand{\ra}{\ensuremath{\rightarrow}}

\newcommand{\hra}{\ensuremath{\hookrightarrow}}
\newcommand{\thra}{\ensuremath{\twoheadrightarrow}}

\newcommand{\GL}[2]{\ensuremath{\mathrm{GL}({#1}, #2)}}

\newcommand{\End}{\ensuremath{\mathrm{End}}}

\newcommand{\ot}{\ensuremath{\otimes}}

\newcommand{\Cstar}{\ensuremath{\text{C}^*}}
\newcommand{\Wstar}{\ensuremath{\text{W}^*}}

\newcommand{\bo}{\ensuremath{\mathcal{B}}}

\newcommand{\FdHilb}{\ensuremath{\mathrm{FdHilb}}}

\newcommand{\cont}{\ensuremath{\mathrm{C}}}

\newcommand{\Linfty}{\ensuremath{{\offinterlineskip \mathrm{L} \hskip -0.3ex ^\infty}}}
\newcommand{\Ltwo}{\ensuremath{{\offinterlineskip \mathrm{L} \hskip -0.3ex ^2}}}

\newcommand{\lspan}{\ensuremath{\mathop{\mathrm{span}}}}

\newcommand{\Ad}{\ensuremath{\mathop{\mathrm{Ad}}}}

\newcommand{\freegrp}[1]{\ensuremath{\mathbb{F}}_{#1}}

\newcommand{\paarpart}{\sqcap}
\newcommand{\baarpartbaustein}{\rotatebox{180}{$\sqcap$}}
\newcommand{\baarpart}{
\mathrel{\vcenter{\offinterlineskip \hbox{$\baarpartbaustein$}}}}

\newcommand{\singleton}{\uparrow}
\newcommand{\doublesingleton}{\singleton\otimes\singleton}

\newcommand{\vierpart}{
\mathrel{\offinterlineskip
\hskip0ex\hbox{$\sqcap$}\hskip -.44ex\hbox{$\sqcap$} \hskip -0.44ex\hbox{$\sqcap$}}}

\newcommand{\primarypart}{
\mathrel{\vcenter{\offinterlineskip
\hbox{$\baarpart$} \vskip -1.3ex \hbox{\hskip1.3ex$/$\hskip-1.2ex$-$} \vskip -1.2ex \hbox{\hskip2.2ex $\paarpart$}}}}

\newcommand{\midmid}{
\mathrel{\vcenter{\offinterlineskip
\vskip -0.2ex \hbox{$\shortmid$} \vskip -0.75ex \hbox{$\shortmid$}}}}

\newcommand{\halflibpart}{
\mathrel{\offinterlineskip
\hbox{$\bigtimes$}\hskip -1.55ex \hbox{$\midmid$} \hskip 1ex}}

\usepackage{calc}
\newcounter{PartitionDepth}
\newcounter{PartitionLength}

\newcommand{\uppartii}[3]{
 \begin{picture}(#3,#1)
 \setcounter{PartitionLength}{#3-#2}
 \setcounter{PartitionDepth}{#1}
 \put(#2,0){\line(0,1){#1}}     
 \put(#3,0){\line(0,1){#1}}
 \put(#2,\thePartitionDepth){\line(1,0){\thePartitionLength}}
 \end{picture}}
\newcommand{\uppartiii}[4]{
 \begin{picture}(#4,#1)
 \setcounter{PartitionLength}{#4-#2}
 \setcounter{PartitionDepth}{#1}
 \put(#2,0){\line(0,1){#1}}
 \put(#3,0){\line(0,1){#1}}
 \put(#4,0){\line(0,1){#1}}
 \put(#2,\thePartitionDepth){\line(1,0){\thePartitionLength}} 
 \end{picture}}

\newcommand{\diag}{\ensuremath{\mathrm{diag}}}

\begin{document}

\author{Sven Raum and Moritz Weber}

\begin{center}
{\LARGE \bf {Easy quantum groups and quantum subgroups of a semi-direct product quantum group}}

\bigskip

{\sc by Sven Raum$^{(1)}$ and Moritz Weber
  \setcounter{footnote}{1}
  \footnotetext{Supported by KU Leuven BOF research grant OT/08/032 and by ANR Grant NEUMANN}
}
\end{center}

\renewcommand{\thefootnote}{}

\begin{abstract}
  We consider homogeneous compact matrix quantum groups whose fundamental corepresentation matrix has entries which are partial symmetries with central support.  We show that such quantum groups have a simple presentation as semi-direct product quantum groups of a group dual quantum group by an action of a permutation group.  This general result allows us to completely classify easy quantum groups with the above property by certain reflection groups.  We give four applications of our result. First, there are uncountably many easy quantum groups.  Second, there are non-easy homogeneous orthogonal quantum groups.  Third, we study operator algebraic properties of the hyperoctahedral series.  Finally, we prove a generalised de Finetti theorem for easy quantum groups in the scope of this article.
\end{abstract}

\section{Introduction}
\label{sec:introduction}

Easy quantum groups were introduced by Banica and Speicher in \cite{banicaspeicher09}, in the context of Wang's universal quantum groups \cite{wang95}.  Via Speicher's partitions, this class of quantum groups has a natural link to free probability theory.  Together with \cite{raumweber13-complete-classification-2}, the present paper completes the classification of easy quantum groups, which was started in \cite{banicaspeicher09,banicacurranspeicher09_2,weber13-classification}.

Compact quantum groups were introduced by Woronowicz in \cite{woronowicz87,woronowicz98}.  As they have a natural set of axioms and satisfy a version of Tannaka-Krein duality \cite{woronowicz88}, they are by now an established generalisation of compact groups to a non-commutative setting.  There are mainly three sources of examples of compact quantum groups.  First, there are $q$-deformations of compact simple Lie groups described in \cite{jimbo85,drinfeld86,woronowicz88,rosso90-version-duke}.  Second, motivated by a question of Connes, Wang described the quantum symmetry group of $n$ points in \cite{wang98}.  This later led Goswami to define quantum isometry groups in \cite{goswami09}.  Third, Banica and Speicher defined a combinatorial class of quantum groups, which they called easy quantum groups \cite{banicaspeicher09}.  This article is concerned with the latter class of quantum groups.

In 1995 Wang introduced the universal orthogonal and the universal unitary quantum groups $O_n^+$ and $U_n^+$ \cite{wang95}. Later in 1998 he introduced the quantum permutation groups $\rS_n^+$ \cite{wang98}.  After work of Banica and others \cite{banica96,banica99,banicabichoncollins07} introducing the free hyperoctahedral quantum group $H_n^+$, it became clear that all these quantum groups should be considered as a result of a ``liberation process'' of classical groups, which was formalised in \cite{banicaspeicher09}.  It resembles the passage from classical probability theory to free probability theory via Speicher's partitions.  The combinatorial point of view taken in \cite{banicaspeicher09} naturally gave rise to a new class of quantum groups which are described by Speicher's partitions.  Banica and Speicher called them \emph{easy quantum groups}.  By their very definition easy quantum groups are related to free probability theory.  This intuition was confirmed by several de Finetti type theorems, identifying easy quantum groups as the correct class of symmetries of certain distributions \cite{koestlerspeicher09-de-finetti,banicacurranspeicher12-finetti}.  In \cite{banicaspeicher09,weber13-classification} the classification of two subclasses of easy quantum groups, namely easy groups and free quantum groups (those easy quantum groups corresponding to non-crossing partitions), was settled.  Moreover, in \cite{banicacurranspeicher09_2,weber13-classification} further classification results for easy quantum groups were obtained, showing that there are exactly 13 easy quantum groups outside a family called \emph{hyperoctahedral easy quantum groups}.  At the same time, it was shown that there is at least a countable number of hyperoctahedral easy quantum groups, leaving their complete classification as an open problem.

We complete the classification of easy quantum groups by classifying hyperoctahedral easy quantum groups in this paper and in \cite{raumweber13-complete-classification-2}. 

We identify a dividing line between easy quantum groups which are quantum subgroups of the semi-direct product quantum groups $\Cstar(\ZZ_2^{*n}) \Join \cont(\rS_n)$ and those which are not.  While the structure of the former class of easy quantum groups is governed by algebraic considerations, the study of the latter class remains of combinatorial nature.  Indeed, in \cite{raumweber13-complete-classification-2} we classify by combinatorial arguments the remaining hyperoctahedral easy quantum groups which are not contained in $\Cstar(\ZZ_2^{*n}) \Join \cont(\rS_n)$.  It turns out that there are only countably many of them.  In contrast, in the present paper we completely classify homogeneous quantum subgroups of $\Cstar(\ZZ_2^{*n}) \Join \cont(\rS_n)$ and only afterwards we identify easy quantum groups among them.  A quantum group is called homogeneous, if it contains the permutation group $\rS_n$ as a quantum subgroup (see Section \ref{sec:homogeneou-quantum-groups} for a precise definition).  The homogeneity assumption is natural for any classification result for compact quantum groups, since every Lie group admits many embeddings into $O_n$ -- these subgroups have to be excluded to expect reasonable classification results.  The algebraic approach of this paper allows us to understand the easy quantum groups in question more profoundly, as is demonstrated by several applications.  Moreover, it is a complete classification result for homogeneous quantum subgroups of a naturally defined quantum group.  

 Already in the preprint \cite{raumweber12}, we obtained some of the results presented in this paper by means of combinatorial methods.  The present article however achieves a complete description of the quantum groups treated there, making use of algebraic methods.  We refer to \cite{raumweber13-combinatorial-approach} for a concise presentation of the combinatorial approach.

Let us describe the results of this paper in more detail.  We first characterise homogeneous compact matrix quantum groups $(A,u)$ such that all $u_{ij}^2$ are central projections in $A$.  A compact matrix quantum group is called homogeneous, if it admits a morphism $A \ra \cont(\rS_n)$.  The compact matrix quantum groups arising in the next theorem are called semi-direct product quantum groups (see Section \ref{sec:semi-direct-products}).  We denote them by $\Cstar(\Gamma) \Join \cont(\rS_n)$.  They generalise the compact groups $\wh H \rtimes \rS_n$, where $H$ is a discrete abelian group carrying an action of $\rS_n$ by group automorphisms.
\begin{thmstar}[See Theorem \ref{thm:classification-of-qsubgroups}]
  \label{thm:structure}
  Let $(A,u)$ be a homogeneous compact matrix quantum group such that $u_{ij}^* = u_{ij}$ and $u_{ij}^2$ is a central projection in $A$ for all $i,j \in \{1, \dotsc, n\}$.  Then there is a quotient $\ZZ_2^{*n} \thra \Gamma$ whose kernel is invariant under the natural action of $\rS_n$ such that $A$ is a version of $\Cstar(\Gamma) \Join \cont(\rS_n)$.  In particular, we have $A \cong \Cstar(\Gamma) \Join \cont(\rS_n)$ if $A$ is in its maximal version.
\end{thmstar}

Applying the previous classification to hyperoctahedral easy quantum groups, we obtain the following structural result.  A strongly symmetric reflection group is a quotient $\ZZ_2^{*n} \ra \Gamma$ whose kernel is invariant under any \emph{identification of letters}  (See Definition \ref{def:strong-symmetric-semigroup}).

\begin{thmstar}[See Theorem \ref{thm:classification-easy-quantum-groups-bicrossed-product-case}]
  \label{thm:classification-easy-quantum-groups-bicrossed-product-case:introduction}
Let $(A,u)$ be an easy quantum group associated with the category of partitions $\cC$.  Assume that $\primarypart \in \cC$.  Then $A \cong \Cstar(\Gamma) \Join \cont(\rS_n)$ for some strongly symmetric reflection group $\Gamma$ on $n$ generators.  Denoting the generators of $\Gamma$ by $g_1, \dotsc, g_n$ and the fundamental corepresentation of $\cont(\rS_n)$ by $(p_{ij})$, the fundamental corepresentation of $A$ is identified with $(u_{g_i}p_{ij})$.

Moreover, every strongly symmetric reflection group arises this way.
\end{thmstar}

Motivated by this result, we call categories of partitions that contain $\primarypart$ \emph{group-theoretical} categories of partitions.  The associated easy quantum groups are called \emph{group-theoretical} easy quantum groups.

Our results allow us to answer some open questions on easy quantum groups.  First of all we show that there are uncountably many easy quantum groups.

\begin{thmstar}[See Theorem \ref{thm:uncountably-many-easy-quantum-groups}]
  \label{thm:uncountably-many-easy-quantum-groups:introduction}
  There are uncountably many pairwise non-isomorphic easy quantum groups.
\end{thmstar}

It was an open problem to find an intermediate quantum group $O_n^+ \supset G \supset \rS_n$ which is not an easy quantum group.  Our classification allows us to find such an example.

\begin{thmstar}[See Theorem \ref{thm:non-easy-quantum-group}]
  \label{thm:non-easy-quantum-group:introduction}
  There is an example of an intermediate quantum group $O_n^+ \supset G \supset \rS_n$ such that $G$ is not an easy quantum group.
\end{thmstar}

In \cite{banicacurranspeicher09_2}, two series of easy quantum groups were introduced. $H_n^{(s)}$ and $H_n^{[s]}$, $s \in \NN \cup \{\infty\}$ are called the hyperoctahedral series and the higher hyperoctahedral series, respectively.  These quantum groups fit into the framework of this article, which allows us to identify them explicitly.  In the light of present interest in operator algebraic properties of free quantum groups, we study the hyperoctahedral series from this point of view.

As a last application of our classification result in Theorem \ref{thm:classification-easy-quantum-groups-bicrossed-product-case:introduction}, we prove a uniform de Finetti theorem for all group-theoretical hyperocatahedral quantum groups in Section \ref{sec:de-finetti-theorems}.  It recovers the de Finetti theorem for $H_n^*$ of Banica, Curran and Speicher \cite{banicacurranspeicher12-finetti}.

\subsection*{Acknowledgements}
The first author thanks Roland Speicher for inviting him to Saarbr\"ucken, where this work was initiated in June 2012.  Both authors thank Stephen Curran for pointing out an error at an early stage of our work on the preprint \cite{raumweber12}, of which this paper is a continuation.  We are very grateful to Teodor Banica and Adam Skalski for useful discussions on the same preprint.  Finally we want to say a warm thank you to Julien Bichon for pointing out a missing hypothesis in Theorem \ref{thm:classification-of-qsubgroups}.

\section{Preliminaries}
\label{sec:preliminaries}

\subsection{Compact matrix quantum groups}
\label{sec:compact-matrix-quantum-groups}

In \cite{woronowicz87,woronowicz91} Woronowicz defined \emph{compact matrix quantum groups} (CMQG), which are the correct analogue of compact Lie groups in the setting of his \emph{compact quantum groups} \cite{woronowicz98}.  A compact matrix quantum group is a unital \Cstar-algebra $A$ with an element $u \in \rM_n(A)$ such that
\begin{itemize}
\item $A$ is generated by the entries of $u$,
\item there is a *-homomorphism $\Delta: A \ra A \ot A$ called \emph{comultiplication} which satisfies $\Delta(u_{ij}) = \sum_k u_{ik} \ot u_{kj}$ for all $1 \leq i,j \leq n$ and
\item $u$ is unitary and its transpose $u^t$ is invertible.
\end{itemize}
The matrix $u$ is called the \emph{fundamental corepresentation} of $(A, u)$.  A morphism between CMQGs $A$ and $B$ is a morphism $\phi:A \ra B$ of the underlying \Cstar-algebras such that $(\phi \ot \id)(u_A)$ is conjugate with $u_B$ by some element in $\GL{n}{\CC}$.  Two CMQGs are called \emph{similar} if they are isomorphic as \Cstar-algebras via a morphism of CMQGs.  They are \emph{isomorphic} if they are isomorphic as \Cstar-algebras via a morphism preserving the fundamental corepresentation.

An important feature of CMQG is the existence of a \emph{Haar state} proved by Woronowicz.
\begin{theorem}[See \cite{woronowicz87}]
  Let $(A,u)$ be a CMQG with comultiplication $\Delta$.  Then there is a unique state, the Haar state, $\phi \in A^*$ such that
  \begin{equation*}
    (\id \ot \phi) \circ \Delta(x) = \phi(x)1 = (\phi \ot \id) \circ \Delta(x)
  \end{equation*}
  for all $x \in A$.
\end{theorem}

A CMQG group $(A,u)$ is called \emph{orthogonal}, if its fundamental corepresentation has self-adjoint entries.  We write $\overline{u} = (u_{ij}^*)$ for the entrywise adjoint of $u$.  One can show that the Haar state $\phi$ of an orthogonal quantum group $(A,u)$ is \emph{tracial}, that is $\phi(xy) = \phi(yx)$ for all $x, y \in A$.  All compact matrix quantum groups considered in this article are orthogonal.

\subsection{Different versions of quantum groups}
\label{sec:versions-of-quantum-groups}

Every CMQG $(A,u)$ contains the \emph{algebra of polynomial functions} $\mathrm{Pol}(A) = *-\mathrm{alg}(u_{ij}, i,j \in \{1, \dotsc, n\})$.  Also the universal enveloping \Cstar-algebra of $\mathrm{Pol}(A)$ is a CMQG, which is called the \emph{maximal version} of $A$.  We say that $A$ is in its maximal version if it is isomorphic to its maximal version.  We write $\cont(G)$ for the maximal version of a CMQG, thinking of $G$ as the quantum group.  In this case we also write $\mathrm{Pol}(G) = \mathrm{Pol}(\cont(G))$.  Any CMQG whose maximal version is isomorphic with $\cont(G)$, is called a version of $\cont(G)$.

The Haar state $\phi$ of $A$ is faithful on $\mathrm{Pol}(G)$. 

For von Neumann algebraic notions we refer the reader to Section \ref{sec:von-Neumann-algebras}.  Taking the weak closure of $\mathrm{Pol}(G)$ in the GNS-representation associated with $\phi$, we obtain the \emph{von Neumann algebraic version} of $A$ denoted by $\Linfty(G)$.  If $A$ is an orthogonal CMQG, then $\phi$ extends to a faithful normal tracial state on $\Linfty(G)$, showing that it is a finite von Neumann algebra.

Let us mention that with every compact matrix quantum group $G$, one can associate a \emph{dual quantum group} $\wh G$, which is a discrete quantum group.  We will only use the notation $\wh G$ and refer to \cite{timmermann08} for more explanation on duality for quantum groups.

\subsection{Tannaka-Krein duality for compact matrix quantum groups}
\label{sec:tannaka-krein-duality}

Let $(A, u)$ be a CMQG with comultiplication $\Delta$.  A \emph{unitary corepresentation matrix} of $A$ is a unitary element $v \in \rM_m(A)$ such that ${\Delta_A(v_{ij}) = \sum_k v_{ik} \ot v_{kj}}$ for all $i,j \in \{1, \dotsc, m\}$.  A morphism between unitary corepresentation matrices $v \in \rM_k(A)$ and $w \in \rM_l(A)$ is a matrix $T \in \rM_{l,k}(\CC)$ such that $T v = w T $.  It is also called an \emph{intertwiner}.  The space of intertwiners between two unitary corepresentation matrices  $v,w$ is denoted by $\Hom(v, w)$.

The \emph{tensor product} of two corepresentation matrices $v \in \rM_k(A)$ and $w \in \rM_l(A)$ is defined as the Kronecker tensor product $(v \ot w)_{(i,i')(j,j')} = v_{ij}w_{i'j'}$.  Denote by $\mathrm{Corep}(A,u)$ the category whose elements are tensor powers $u^{\ot k}$, $k \in \NN$ and whose morphisms are intertwiners.  Then $\mathrm{Corep}(A,u)$ is a \emph{concrete compact tensor \Cstar-category} in the sense of Woronowicz (see \cite{woronowicz87,woronowicz88} or \cite[Chapter 5]{timmermann08}).

We need the following special version of Woronowicz's Tannaka-Krein duality.
\begin{theorem}[See \cite{woronowicz88}]
\label{thm:tannaka-krein-duality}
Any concrete tensor \Cstar-category whose objects are tensor powers of a specified object $x$ is equivalent to the category $\mathrm{Corep}(A,u)$ for some orthogonal compact matrix quantum group $(A,u)$ such that $u$ corresponds to $x$ under this equivalence of categories.  Two orthogonal compact matrix quantum groups $(A,u)$ and $(B,v)$ are similar if and only if the categories $\mathrm{Corep}(A,u)$ and $\mathrm{Corep}(B,v)$ are equivalent over $\FdHilb$ via an equivalence that sends the isomorphism class of $u$ to the isomorphism class of $v$.
\end{theorem}

\subsection{Easy quantum groups}
\label{sec:easy-quantum-groups}

In order to describe corepresentation categories of compact quantum matrix groups combinatorially, Banica and Speicher introduced the notions of a \emph{category of partitions} and of \emph{easy quantum groups} \cite{banicaspeicher09}.  Alternatively, we speak about \emph{partition quantum groups}.  A \emph{partition} $p$ is given by $k$ upper points and $l$ lower points which may be connected by lines.  This way, the set of $k+l$ points is partitioned into several \emph{blocks}. We write a partition as a diagram in the following way:
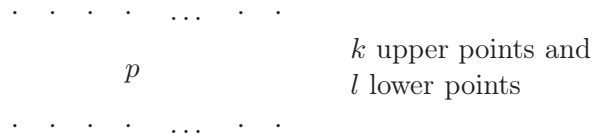
\setlength{\unitlength}{0.5cm}
\begin{center}
  \begin{picture}(14,3)
    \put(0,0){$\cdot$}
    \put(1,0){$\cdot$}
    \put(2,0){$\cdot$}
    \put(3,0){$\cdot$}
    \put(4.2,0){$\dotsc$}
    \put(6,0){$\cdot$}
    \put(7,0){$\cdot$}

    \put(3,1.5){$p$}

    \put(0,3){$\cdot$}
    \put(1,3){$\cdot$}
    \put(2,3){$\cdot$}
    \put(3,3){$\cdot$}
    \put(4.2,3){$\dotsc$}
    \put(6,3){$\cdot$}
    \put(7,3){$\cdot$}
    
    \put(9,1.5){\begin{minipage}{3.5cm} $k$ upper points and \\ $l$ lower points \end{minipage}}
  \end{picture}
\end{center}

Two examples of such partitions are the following diagrams.
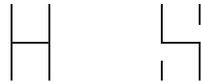
\setlength{\unitlength}{0.5cm}
\begin{center}
  \begin{picture}(6,3)
    \put(0,0){\line(0,1){2}}
    \put(1,0){\line(0,1){2}}
    \put(0,1){\line(1,0){1}}

    \put(4,0){\line(0,1){0.5}}
    \put(5,0){\line(0,1){1}}
    \put(4,2){\line(0,-1){1}}
    \put(5,2){\line(0,-1){0.5}}
    \put(4,1){\line(1,0){1}}
  \end{picture}
\end{center}
In the first example, all four points are connected and the partition consists only of one block.  In the second example, the left upper point and the right lower point are connected, whereas neither of the two remaining points is connected to any other point.

The set of partitions on $k$ upper and $l$ lower points is denoted by $\cP(k,l)$, and the set of all partitions is denoted by $\cP$.  We also write $\cP(k)$ for the set of all partitions which have only $k$ upper points and no lower points.  We will also use \emph{labelled partitions}, i.e. partitions whose points are labelled by natural numbers.  Vice versa, we can associate with an index $i = (i_1, \dotsc, i_k) \in \{1, \dotsc, n\}^k$ the partition $\ker(i) \in \cP(k)$ whose blocks are exactly $\{j \amid i_j = i_0\}$, where $i_0$ runs through $\{1, \dotsc, n\}$.

There are the natural operations \emph{tensor product} ($p\ot q$), \emph{composition} ($pq$), \emph{involution} ($p^*$) and \emph{rotation} on partitions (see \cite[Definition 1.8]{banicaspeicher09} or \cite[Definition 1.4]{weber13-classification}).  If $p \in \cP(k,l)$, $q \in \cP(k',l')$, then $p \ot q \in \cP(k+k', l+l')$ is the partition obtained by writing $p$ and $q$ next to each other.  If $p \in \cP(k,l)$, $q \in \cP(l,m)$, then $qp$ is the partition obtained by writing $p$ above $q$, connecting them along the $l$ intermediate points and deleting all closed strings which are not connected to any of the $k + m$ remaining points.  The involution $p^* \in \cP(l,k)$ of $p \in \cP(k,l)$ is obtained by turning $p$ upside down.  A basic rotation of $p \in \cP(k,l)$ is the partition in $\cP(k-1, l+1)$ or in $\cP(k+1, l-1)$ arising after one turns the first leg of the upper row and puts it in front of the first leg of the lower row -- or vice versa.  Now, a rotation of $p$ is a partition which arises after applying multiple basic rotations to $p$.

A collection $\cC$ of subsets $\cC(k,l) \subset \cP(k,l)$, $k,l \in \NN$ is called a \emph{category of partitions} if it is closed under these operations and if it contains the pair partition $\sqcap$ (see \cite[Definition 6.1]{banicaspeicher09} or \cite[Definition 1.4]{weber13-classification}).

Given a partition $p \in \cP(k,l)$ and two multi-indexes $(i_1, \dotsc, i_k)$, $(j_1, \dotsc, j_l)$, we can label the diagram of $p$ with these numbers.  The upper and the lower row both are labelled from left to right and we define
\begin{equation*}
  \delta_p(i,j)
  =
  \begin{cases}
    1 & \text{if } p \text{ connects only equal indexes,} \\
    0 & \text{if there is a block of } p \text{ connecting unequal indexes} \eqstop
  \end{cases}
\end{equation*}
For every $n \in \NN$, one associates a map $T_p: (\CC^n)^{\ot k} \ra (\CC^n)^{\ot l}$ with $p$ which is defined by
\begin{equation*}
  T_p(e_{i_1} \ot \dotsm \ot e_{i_k})
  =
  \sum_{j_1, \dotsc, j_l \in \{1, \dotsc, n\}} \delta_p(i, j) \cdot e_{j_1} \ot \dotsm \ot e_{j_l}
  \eqstop
\end{equation*}

\begin{definition}[Definition 6.1 of \cite{banicaspeicher09} or Definition 2.1 of \cite{banicacurranspeicher09_2}]
  An orthogonal compact matrix quantum group $(A,u)$ is called \emph{easy}, if there is a category of partitions $\cC$ given by sets $\cC(k,l) \subset \cP(k,l)$, for all $k,l \in \NN$ such that
  \begin{equation*}
    \Hom(u^{\ot k}, u^{\ot l}) = \lspan \{T_p \amid p \in \cC(k,l) \}
    \eqstop
  \end{equation*}
\end{definition}  

We can apply Theorem \ref{thm:tannaka-krein-duality}, in order to obtain the following one-to-one correspondence between categories of partitions and easy quantum groups.  It is the basis of combinatorial investigations of easy quantum groups.
\begin{theorem}[See \cite{banicaspeicher09}]
\label{thm:classification-easy-quantum-groups-by-partitions}
  There is a bijection between
  \begin{itemize}
  \item categories of partitions $\cC$ and
  \item families of easy quantum groups $G_\cC(n)$, $n \in \NN$, up to similarity.
  \end{itemize}
\end{theorem}

\subsubsection{Homogeneous quantum groups}
\label{sec:homogeneou-quantum-groups}

The permutation groups $\rS_n$ arise as easy quantum groups associated with the category of all partitions.  In the framework of compact matrix quantum groups they can be presented as
\begin{equation*}
  \cont(\rS_n) \cong \Cstar(p_{ij}, 1 \leq i,j \leq n \amid p \text{ is unitary, } p_{ij} \text{ are commuting projections})
  \eqcomma
\end{equation*}
where $p = (p_{ij})$ is the fundamental corepresentation.  

A compact matrix quantum group $(A, u)$ is called \emph{homogeneous} if there is a morphism of compact matrix quantum groups $A \ra \cont(\rS_n)$.  Put differently, $\rS_n$ is a quantum subgroup of the compact quantum group described by $A$.

Note that all easy quantum groups are homogeneous, since every category of partitions is contained in the category of all partitions.

\subsubsection{Hyperoctahedral quantum groups}
\label{sec:hyperoctahedral-easy-quantum-groups}

Properties of an easy quantum group are -- in principle -- completely described by their category of partitions.  Let us recall some elementary instances of this fact.

A category of partitions $\cC$ is called \emph{hyperoctahedral} if the four block $\vierpart$ is in $\cC$, but the double singleton $\singleton \ot \singleton$ is not in $\cC$.  

\begin{proposition}
  \label{prop:description-hyperoctahedral-quantum-groups}
  Let $(A,u)$ be an easy quantum group whose category of partitions is $\cC$.
  \begin{enumerate}
  \item The entries $u_{ij}$ of $u$ are partial isometries if and only if $\vierpart \in \cC$.
  \item The elements $u_{ij}^2$ are central projections in $A$ if and only if $\primarypart \in \cC$.
  \item The corepresentation matrix $u$ is irreducible if and only if $\doublesingleton \notin \cC$.
  \end{enumerate}
\end{proposition}
\begin{proof}
  By Tannaka-Krein duality, $A$ is the universal \Cstar-algebra generated by the entries $u_{ij}$ of its fundamental corepresentation which satisfy $T_p u^{\ot k} = u^{\ot l} T_p$ for all $p \in \cC(k,l)$.  Now item (i) follows by comparing coefficients of $T_p$ and $u^{\ot 4} T_p$ for $p = \vierpart$.  Similarly item (ii) follows by comparing coefficients of $T_p u^{\ot 3}$ and $ u^{\ot 3} T_p$ for $p = \primarypart$.

  In order to see (iii), note that $u$ is irreducible if and only if $\dim \Hom(1, u \ot u) = 1$.  Since $T_\sqcap \in \Hom(1, u \ot u)$ and $|\cP(2)| = 2$, it follows that $u$ is irreducible, if and only if $T_{\doublesingleton} \notin \Hom(1, u \ot u)$.  The latter is equivalent to $\doublesingleton \notin \cC$.
\end{proof}

\subsection{Semi-direct product quantum groups}
\label{sec:semi-direct-products}

Let $\Gamma = \langle g_1, \dotsc, g_n \rangle$ be a finitely generated group and assume that $\rS_n$ acts on $\Gamma$ by permuting $g_1, \dotsc, g_n$.  Then $\rS_n$ also acts on the maximal group \Cstar-algebra $\Cstar(\Gamma)$.  We describe a semi-direct product construction implementing this action.  It is a very simple special case of the bicrossed-product constructions described in \cite{majid90,majid91,klimykschmudgen97,vaesvainerman03}.

Recall that the semi-direct product of two groups $G$ and $H$, where $H$ acts by group automorphisms $(\alpha_h)_{h \in H}$ on $G$ is $G \times H$  as a set whose multiplication is given by $(g_1, h_1)(g_2,h_2) = (g_1 \alpha_{h_1}(g_2), h_1h_2)$.  This picture should be kept in mind in order to understand the semi-direct product construction which is described in the following proposition.

  Denote by $(p_{ij})$ the fundamental corepresentation of $\cont(\rS_n)$.  The CMQG described in the next proposition is called the semi-direct product of $\Cstar(\Gamma)$ and $\cont(\rS_n)$ and it is denoted by $\Cstar(\Gamma) \Join \cont(\rS_n)$.

\begin{proposition}
  \label{prop:semi-direct-products}
  The \Cstar-algebra $\Cstar(\Gamma) \ot \cont(\rS_n)$ is a CMQG with fundamental corepresentation $(u_{g_i}p_{ij})$.  Its comultiplication is given by $\Delta(u_{g_i}p_{ij}) = \sum_k u_{g_i} p_{ik} \ot u_{g_k}p_{kj}$ for all $i,j \in \{1, \dotsc, n\}$.
\end{proposition}
\begin{proof}
  We first show that $(u_{g_i} p_{ij})$ and $(u_{g_i}^* p_{ij})$ are unitaries in $\rM_n(\Cstar(\Gamma) \ot \cont(G))$.  We can use the relation $\sum_k p_{ki} = 1$ so as to obtain 
  \begin{equation*}
    \sum_k p_{ki} u_{g_k}^* u_{g_k} p_{kj}
    =
    \delta_{ij} \sum_k p_{ki}
    =
    \delta_{ij} 1
    \eqstop
  \end{equation*}
  So $(u_{g_i} p_{ij})$ is an isometry.  Similarly it follows that $(u_{g_i}^* p_{ij})$ is an isometry.  As $\Cstar(\Gamma) \ot \cont(\rS_n)$ has a faithful tracial state, this implies that both matrices are unitary.

  Since $\sum_j u_{g_i} p_{ij} = u_{g_i}$ and $(u_{g_i}p_{ij})^2 = p_{ij}$, it follows that the entries of $(u_{g_i}p_{ij})$ generate $\Cstar(\Gamma) \ot \cont(\rS_n)$.
  Finally let us show that there is a comultiplication on $\Cstar(\Gamma) \ot \cont(\rS_n)$ which admits $(u_{g_i}p_{ij})$ as a fundamental corepresentation.  The right $\rS_n$-action $\sigma(g_i) = g_{\sigma^{-1}(i)}$ gives rise to a right $\rS_n$-action on $\Cstar(\Gamma)$.  Denote by $\delta: \Cstar(\Gamma) \ra \cont(\rS_n) \ot \Cstar(\Gamma)$ the corresponding coaction satisfying $\delta(u_{g_i}) = \sum_k p_{ik} \ot u_{g_k}$.  Denote by $\Delta_{\rS_n}$ and $\Delta_\Gamma$ the comultiplication of $\cont(\rS_n)$ and $\Cstar(\Gamma)$, respectively.  Then $\Delta = ((\id \ot \delta \ot \id \ot \id) \circ (\Delta_\Gamma \ot \Delta_{\rS_n}))_{12324}$ from $\Cstar(\Gamma) \ot \cont(\rS_n)$ to $\Cstar(\Gamma) \ot \cont(\rS_n) \ot \Cstar(\Gamma) \ot \cont(\rS_n)$ satisfies
  \begin{align*}
    \Delta(u_{g_i}p_{ij})
    & =
    ((\id \ot \delta \ot \id \ot \id)(\sum_k u_{g_i} \ot u_{g_i} \ot p_{ik} \ot p_{kj}))_{12324} \\
    & =
    (\sum_{k,l} u_{g_i} \ot p_{il} \ot u_{g_l} \ot p_{ik} \ot p_{kj})_{12324}  \\
    & =
    \sum_k u_{g_i}p_{ik} \ot u_{g_k}p_{kj}
    \eqcomma
  \end{align*}
  for all $i,j \in \{1, \dotsc, n\}$.  This finishes the proof.
\end{proof}

\begin{example}
  Consider the permutation action of $\rS_n$ on the natural generators of $\ZZ^{\oplus n}$.  Note that $\Cstar(\ZZ^{\oplus n}) \cong \cont(\mathbb{T}^n)$ as compact quantum groups by Pontryagin duality.  We have $\Cstar(\ZZ^{\oplus n}) \Join \cont(\rS_n) \cong \cont(\mathbb{T}^n \rtimes \rS_n)$ as compact matrix quantum groups.  Similarly, we have $\Cstar(\ZZ_2^{\oplus n}) \Join \cont(\rS_n) \cong \cont(\ZZ_2^{\oplus n} \rtimes \rS_n)$.
\end{example}

\subsection{Diagonal subgroups}
\label{sec:diagonal-subgroups}

If $(A,u)$ is a CMQG whose fundamental corepresentation is a diagonal matrix, then its diagonal entries $u_{ii}$, $i \in \{1, \dotsc, n\}$ are unitary.  Let $\Gamma$ be the group generated by these diagonal entries.  Then $A$ is a \Cstar-algebra completion of the group algebra $\CC \Gamma$.  Using this fact, one can associate with any CMQG a canonical discrete group with a fixed set of generators.
\begin{definition}
  \label{def:diagonal-subgroup}
  Let $(A,u)$ be a CMQG.  Denote by $\pi:A \ra B$ the quotient of $A$ by the relations $u_{ij} = 0$ for all $i \neq j$.  Let $\Gamma$ be the group generated by the elements $g_i = \pi(u_{ii})$.  Then $\Gamma$ together with the generators $g_1, \dotsc, g_n$ is called the \emph{diagonal subgroup} of $(A,u)$.  We denote it by $\diag(A,u)$.
\end{definition}

We will use the following proposition, guaranteeing that a CMQG in its maximal version gives rise to the maximal group \Cstar-algebra of its diagonal subgroup.
\begin{proposition}
  \label{prop:diagonal-subgroups-of-maximal-versions}
  Let $(A,u)$ be a CMQG in its maximal version and let $\Gamma$ be the diagonal subgroup of $(A,u)$.  Then the \Cstar-algebra $A/ (u_{ij} = 0 \text{ for all } i \neq j) \cong \Cstar(\Gamma)$ is in its maximal version.
\end{proposition}
\begin{proof}
  Since $(A,u)$ is in its maximal version, it is the universal enveloping \Cstar-algebra of its $\text{*-subalgebra}$ $\text{*-}\mathrm{alg}(u_{ij} \amid i,j \in \{1, \dotsc, n\})$.  Denote by $\pi:A \ra B$ the quotient of $A$ by the relations $u_{ij} = 0$ for all $i \neq j$.  Then $B$ is the universal enveloping \Cstar-algebra of the *-algebra $\text{*-}\mathrm{alg}(\pi(u_{ii}) \amid i \{1, \dotsc, n\}) \cong \CC \Gamma$.  So $B \cong \Cstar(\Gamma)$.
\end{proof}

\subsection{Operator algebraic properties of quantum groups}
\label{sec:operator-algebraic-properties}

In Section \ref{sec:hyperoctahedral-series} we will describe certain operator algebraic properties of easy quantum groups.  Let us briefly describe von Neumann algebras and some of their properties.  We refer the reader to the book \cite{brownozawa08} for more details on approximation properties for operator algebras and to \cite{dixmier57} for an introduction to von Neumann algebras.

\subsubsection{Von Neumann algebras}
\label{sec:von-Neumann-algebras}

A separable von Neumann algebra is a strongly closed, unital *-subalgebra $M \subset \bo(H)$ for some (complex) separable Hilbert space $H$.  All von Neumann algebras in this article are separable.  We say that $M$ is finite, if there is a faithful normal tracial state on $M$, i.e. a $\sigma$-strongly continuous functional $\tau: M \ra \CC$ such that
\begin{itemize}
\item $\tau(x^*x) \geq 0$ for all $x \in M$,
\item if $\tau(x^*x) = 0$ then $x = 0$ and
\item $\tau(xy) = \tau(yx)$ for all $x,y \in M$.
\end{itemize}

If $M$ has a normal faithful tracial state $\tau$, then $(x,y) \mapsto \tau(y^*x)$ defines an inner product on $M$.  The Hilbert space completion of $M$ with respect to this inner product is denoted by $\Ltwo(M, \tau)$.  The action of $M$ on itself by left multiplication can be extended to a representation on $\Ltwo(M, \tau)$ called the \emph{GNS-representation} of $M$ associated with $\tau$.  We say that an inclusion of finite von Neumann algebras $N \subset M$ with faithful normal tracial state $\tau$ has \emph{finite index} if the commutant $N' = {\{x \in \bo(\Ltwo(M, \tau)) \amid xy = yx \text{ for all } y \in N\}}$ is a finite von Neumann algebra.  If $M$ is finitely generated as a left $N$ module, then $N \subset M$ has finite index for all traces on $M$.

\subsubsection{Amenability}
\label{sec:amenability}

The notion of amenability for discrete groups goes back to the work of von Neumann in \cite{vonneumann29}.  All abelian groups are amenable, while the basic example of a non-amenable group is a free group $\freegrp{n}$.

  Based on his fundamental work in \cite{connes76}, Connes introduced the notion of amenability for von Neumann algebras in \cite{connes78}.  We only need the following three facts about amenability, which follow from Connes's work \cite{connes76}.
  \begin{itemize}
  \item Every von Neumann subalgebra of an amenable von Neumann algebra is amenable. 
  \item The group von Neumann algebra of a discrete group $G$ is amenable, if and only if $G$ is amenable.
  \item Every finite index extension of an amenable von Neumann algebra is amenable.
  \end{itemize}

  In \cite{ruan96,tomatsu06}, Ruan and Tomatsu describe amenability of quantum groups.  We need the following very special case of their work, which already appeared in the article of Ruan.
\begin{theorem}[{See \cite[Theorem 4.5]{ruan96}}]
  \label{thm:amenability-for-quantum-groups}
  Let $G$ be a compact quantum subgroup of $O_n^+$.  Then the discrete dual $\wh G$ is an amenable quantum group, if and only $\Linfty(G)$ is an amenable von Neumann algebra.
\end{theorem}

An important notion in von Neumann algebra theory is \emph{strong solidity} introduced in \cite{ozawapopa10-cartan1}.  A von Neumann algebra $M$ is called diffuse, if there are no minimal projections in $M$.  We call $M$ strongly solid if for all amenable, diffuse von Neumann subalgebras $A \subset M$, the normaliser $\cN_M(A)'' = {\mathrm{vN}(u \in \cU(M) \amid uAu^* = A)}$ is also amenable.

\subsubsection{Haagerup property}
\label{sec:haagerup-property}

The Haagerup property for groups goes back to \cite{haagerup79}, where Haagerup proved that free groups have the Haagerup property.  More generally, this property is preserved under free products.
\begin{theorem}[{See \cite{cherixcowlingjolissaintjulgvalette01}}]
  \label{thm:free-products-preserve-H}
  Let $G_1, G_2$ be groups with the Haagerup property.  Then also $G_1 * G_2$ has the Haagerup property.
\end{theorem}

  Based on a lecture by Connes the Haagerup property was defined in the setting of finite von Neumann algebra by Choda \cite{choda83}.  In particular, she proves that a group has the Haagerup property if and only if its group von Neumann algebra has this property.
  \begin{theorem}[{See \cite{choda83}}]
    \label{thm:H-is-operator-algebraic}
    Let $G$ be a discrete group.  Then $G$ has the Haagerup property if and only if $\rL(G)$ has the Haagerup property.    
  \end{theorem}

  The article \cite{jolissaint02} describes basic properties of the Haagerup property for finite von Neumann algebras.
\begin{theorem}[{See \cite[Theorem 1.1]{jolissaint02}}]
  \label{thm:H-stable}
  Let $N \subset M$ be a finite index inclusion of von Neumann algebras.  Then $N$ has the Haagerup property if and only if $M$ has the Haagerup property.  
\end{theorem}

Recently, in \cite{dawsfimaskalskiwhite13} the notion was systematically investigated in the framework of quantum groups.  We cite a special case of the main result of this article.
\begin{theorem}[{See \cite{dawsfimaskalskiwhite13}}]
  \label{thm:H-for-quantum-groups}
  Let $G$ be a compact quantum subgroup of $O_n^+$.  Then $\wh G$ has the Haagerup property if and only if $\Linfty(G)$ has the Haagerup property as a von Neumann algebra.
\end{theorem}

\subsubsection{The complete matrix approximation property}
\label{sec:CMAP}

The complete metric approximation property (CMAP) for groups goes back to the work of Haagerup in \cite{haagerup79} and de Canni{\`e}re-Haagerup \cite{decannierehaagerup85}.  In \cite{ricardxu06} it is proved that the free product of groups having CMAP still has CMAP.
\begin{theorem}[{See \cite[Theorem 4.13]{ricardxu06}}]
  \label{thm:free-products-preserve-CMAP}
  If $G_1$ and $G_1$ are groups with the CMAP, then also $G_1 * G_2$ hast the CMAP.
\end{theorem}

The von Neumann algebraic analogue of CMAP is called \Wstar-completely contractive approximation property (\Wstar-CCAP).  We state a special case of a result in \cite{haagerup86} (see also \cite{haagerupkraus94}).
\begin{theorem}[{See \cite{haagerup86}}]
  \label{thm:CMAP-is-operator-algebraic}
  Let $G$ be a discrete group.  Then $G$ has CMAP if and only if $\rL(G)$ has the $\Wstar$-CCAP.
\end{theorem}

The following stability result for the \Wstar-CCAP is well known. It can be proved using \cite[Theorem 12.3.13]{brownozawa08}.
\begin{theorem}
  Let $N \subset M$ be a finite index inclusion of von Neumann algebras.  Then $N$ has the \Wstar-CCAP if and only if $M$ has the \Wstar-CCAP.
\end{theorem}

In the context of discrete quantum groups, CMAP was studied as well.
\begin{theorem}[{See \cite{krausruan97}}]
  \label{thm:CMAP-for-quantum-groups}
  Let $G$ be compact quantum subgroup of $O_n^+$.  Then $\wh G$ has the CMAP if and only if $\Linfty(G)$ has the \Wstar-CCAP.
\end{theorem}

\section{Homogeneous quantum subgroups of $\Cstar(\ZZ_2^{*n}) \Join \cont(\rS_n)$}
\label{sec:classification-of-qsubgroups}

Recall from Section \ref{sec:semi-direct-products} that $\Cstar(\ZZ_2^{*n}) \Join \cont(\rS_n)$ denotes the CMQG whose \Cstar-algebra is isomorphic with $\Cstar(\ZZ_2^{*n}) \ot \cont(\rS_n)$ and whose fundamental corepresentation is $(u_{a_i}p_{ij})$.  Here $a_1, \dotsc, a_n$ denote the natural generators of $\ZZ_2^{*n}$ and $(p_{ij})$ is the fundamental corepresentation of $\cont(\rS_n)$.  Note that $(u_{a_i}p_{ij})^2 = p_{ij}$ is a central projection in $\Cstar(\ZZ_2^{*n}) \ot \cont(\rS_n)$ for all $i,j \in \{1, \dotsc, n\}$.  The next theorem tells us in particular that $\Cstar(\ZZ_2^{*n}) \Join \cont(\rS_n)$ is the universal homogeneous quantum group with this property.

\begin{theorem}
  \label{thm:classification-of-qsubgroups}
  Let $(A,u)$ be a homogeneous orthogonal compact matrix quantum group such that $u_{ij}^2$ is a central projection in $A$ for all $i,j \in \{1, \dotsc, n\}$.  Then there is a quotient $\ZZ_2^{*n} \thra \Gamma$ whose kernel is invariant under the natural action of $\rS_n$, such that $A$ is a version of $\Cstar(\Gamma) \Join \cont(\rS_n)$.  In particular, we have $A \cong \Cstar(\Gamma) \Join \cont(\rS_n)$ if $A$ is in its maximal version.
\end{theorem}
\begin{proof}
  First note that if $(A,u)$ is a compact matrix quantum group such that $u_{ij}^2$ is a central projection in $A$ for all $i,j \in \{1, \dotsc, n\}$, then the elements $u_{ij}^2$ are central projections in $\mathrm{Pol}(A)$ and hence the same is true in the maximal version of $(A,u)$.  We may hence assume that $A$ is in its maximal version.

  \textbf{An embedding $\cont(S_n) \ra A$:}
  Since $u_{ij}$ is a self-adjoint partial isometry for all $i,j \in \{1, \dotsc, n\}$ and $u$ is unitary, it follows that $\sum_k u_{ik}^2 = 1 = \sum_k u_{kj}^2$ for all $i,j \in \{1, \dotsc, n\}$.  Since $q_{ij} = u_{ij}^2 \in A$, $i,j \in \{1, \dotsc, n\}$ is also a commuting family of projections, it satisfies the relations of the fundamental corepresentation of $\cont(\rS_n)$.  So there is a morphism $\cont(\rS_n) \ra A$ of compact quantum groups defined by $p_{ij} \mapsto q_{ij}$.  Here $p = (p_{ij})$ denotes the fundamental corepresentation of $\cont(\rS_n)$.  Since $A$ is homogeneous, there is a morphism $A \ra \cont(\rS_n)$ of compact matrix quantum groups, satisfying $u_{ij} \mapsto p_{ij}$.  This shows that $\cont(\rS_n) \ra A$ is an embedding.

  \textbf{Construction of $\Gamma$:}
  Let now $v_i = \sum_j u_{ij}$ for $i \in \{1, \dotsc, n\}$.  Then all $v_i$ are self-adjoint unitaries, since they are obviously self-adjoint and
  \begin{equation*}
    v_i^2 = \sum_{j,k} u_{ij}u_{ik} = \sum_j q_{ij} = 1
    \eqstop
  \end{equation*}
  Denote by $\Gamma$ the diagonal subgroup of $(A,u)$.  By Proposition \ref{prop:diagonal-subgroups-of-maximal-versions} the quotient ${A/(u_{ij} = 0 \text{ for all } i \neq j)}$ appears in its maximal version $\Cstar(\Gamma)$, as $A$ is in its maximal version.  Denote by $\pi:A \ra \Cstar(\Gamma)$ the natural quotient map.  Since $\pi(v_i) = \pi(u_{ii})$ and $v_i^2 = 1$, there is a quotient map $\ZZ_2^{*n} \ra \Gamma$ mapping the $i$-th generator of $\ZZ_2^{*n}$ to $\pi(v_i)$.
  
  \textbf{A *-homomorphism} $\Cstar(\Gamma) \ra A$\textbf{:}
  Denote by $g_i$ the natural generators of $\Gamma$, which satisfy $u_{g_i} = \pi(u_{ii}) = \pi(v_i)$ for all $i \in \{1, \dotsc, n\}$.  We show that there is a map $\Cstar(\Gamma) \ra A$ which maps $u_{g_i}$ to $v_i$.   By universality of $\Cstar(\Gamma)$, it suffices to show that the unitaries $v_i$, $i \in \{1, \dotsc, n\}$ satisfy the relations of $g_i$, $i \in \{1, \dotsc, n\}$.  So assume that $g_{i_1} \dotsm g_{i_l} = e$.  Then $\pi(v_{i_1} \dotsm v_{i_l}) = u_{g_{i_1}} \dotsm u_{g_{i_l}} = 1$.  Let $i_{l + 1}, \dotsc, i_{l'}$ be an enumeration of $\{1, \dotsc, n\} \setminus \{i_1, \dotsc, i_l\}$.  Note that $v_i^2 = 1$ implies $g_i^2 = e$.  So
  \begin{align*}
    \pi(u_{i_1 i_1} \dotsm \, u_{i_l i_l} q_{i_{l+1} i_{l+1}} \dotsm \, q_{i_{l'}i_{l'}}) = 1
  \end{align*}
  holds and implies 
  \begin{equation}
    \label{eq:product-contained-in-ideal}
    u_{i_1 i_1} \dotsm \, u_{i_l i_l} q_{i_{l+1} i_{l+1}} \dotsm \, q_{i_{l'}i_{l'}}
    \quad \in \quad 
    q_{i_1 i_1} \dotsm \, q_{i_l i_l} q_{i_{l+1} i_{l+1}} \dotsm \, q_{i_{l'}i_{l'}}
    +
    \langle q_{ij}, i \neq j \rangle
    \eqcomma
  \end{equation}
  where the last expression denotes the ideal in $A$ which is generated by all $q_{ij}$, $i \neq j$.  Note that indeed  $\langle q_{ij}, i \neq j \rangle = \langle u_{ij}, i \neq j \rangle$, since $v_i q_{ij} = u_{ij}$.  Next note that $\{i_1, \dotsc, i_{l'}\} = \{1, \dotsc, n\}$ implies 
  \begin{equation}
    \label{eq:product-zero}
    \langle q_{ij}, i \neq j \rangle
    \cdot
    q_{i_1 i_1} \dotsm \, q_{i_l i_l} q_{i_{l+1} i_{l+1}} \dotsm \, q_{i_{l'}i_{l'}}
    =
    0
    \eqcomma
  \end{equation}
  because $q_{ij} q_{ij'} = 0 = q_{ij} q_{i'j}$ for all $i \neq i'$ and $j \neq j'$.  Moreover,  since $q_{ij}$ is central in $A$, we have
  \begin{equation}
    \label{eq:product-absorbed}
    (u_{i_1 i_1} \dotsm \, u_{i_l i_l} q_{i_{l+1} i_{l+1}} \dotsm \, q_{i_{l'}i_{l'}})
    \cdot
    (q_{i_1 i_1} \dotsm \, q_{i_l i_l} q_{i_{l+1} i_{l+1}} \dotsm \, q_{i_{l'}i_{l'}})
    =
    u_{i_1 i_1} \dotsm \, u_{i_l i_l} q_{i_{l+1} i_{l+1}} \dotsm \, q_{i_{l'}i_{l'}}
    \eqstop
  \end{equation}
  Multiplying (\ref{eq:product-contained-in-ideal}) with $q_{i_1 i_1} \dotsm \, q_{i_l i_l} q_{i_{l+1} i_{l+1}} \dotsm \, q_{i_{l'}i_{l'}}$, and applying (\ref{eq:product-zero}) and (\ref{eq:product-absorbed}), we see that
  \begin{equation}
    \label{eq:equal-to-projection}
    u_{i_1 i_1} \dotsm \, u_{i_l i_l} q_{i_{l+1} i_{l+1}} \dotsm \, q_{i_{l'}i_{l'}}
    =
    q_{i_1 i_1} \dotsm \, q_{i_l i_l} q_{i_{l+1} i_{l+1}} \dotsm \, q_{i_{l'}i_{l'}}
    \eqstop
  \end{equation}
  Applying $\Delta$ to (\ref{eq:equal-to-projection}), we obtain
  \begin{multline*}
    \sum_{k_1, \dotsc, k_{l'}}
    u_{i_1 k_1} \dotsm \, u_{i_l k_l} q_{i_{l+1} k_{l+1}} \dotsm \, q_{i_{l'}k_{l'}} 
    \ot
    u_{k_1 i_1} \dotsm \, u_{k_l i_l} q_{k_{l+1} i_{l+1}} \dotsm \, q_{k_{l'}i_{l'}} \\
    =
    \sum_{k_1, \dotsc, k_{l'}}
    q_{i_1 k_1} \dotsm \, q_{i_l k_l} q_{i_{l+1} k_{l+1}} \dotsm \, q_{i_{l'}k_{l'}}
    \ot
    q_{k_1 i_1} \dotsm \, q_{k_l i_l} q_{k_{l+1} i_{l+1}} \dotsm \, q_{k_{l'}i_{l'}}
  \end{multline*}
  Multiplying this equation on both sides with $q_{i_1 k_1} \dotsm \, q_{i_l k_l} q_{i_{l+1} k_{l+1}} \dotsm \, q_{i_{l'}k_{l'}} \ot q_{k_1 i_1} \dotsm \, q_{k_l i_l} q_{k_{l+1} i_{l+1}} \dotsm \, q_{k_{l'}i_{l'}}$ for a fixed index $k_1, \dotsc, k_{l'}$, we obtain that for all such indexes
  \begin{multline}
    \label{eq:tensor-product-equality}
    u_{i_1 k_1} \dotsm \, u_{i_l k_l} q_{i_{l+1} k_{l+1}} \dotsm \, q_{i_{l'}k_{l'}} 
    \ot
    u_{k_1 i_1} \dotsm \, u_{k_l i_l} q_{k_{l+1} i_{l+1}} \dotsm \, q_{k_{l'}i_{l'}} \\
    =
    q_{i_1 k_1} \dotsm \, q_{i_l k_l} q_{i_{l+1} k_{l+1}} \dotsm \, q_{i_{l'}k_{l'}}
    \ot
    q_{k_1 i_1} \dotsm \, q_{k_l i_l} q_{k_{l+1} i_{l+1}} \dotsm \, q_{k_{l'}i_{l'}}
  \end{multline}
  holds.  This implies that
  \begin{equation*}
    u_{i_1 k_1} \dotsm \, u_{i_l k_l} q_{i_{l+1} k_{l+1}} \dotsm \, q_{i_{l'}k_{l'}}
    =
    q_{i_1 k_1} \dotsm \, q_{i_l k_l} q_{i_{l+1} k_{l+1}} \dotsm \, q_{i_{l'}k_{l'}}
  \end{equation*}
  for all $k_1, \dotsc, k_{l'} \in \{1, \dotsc, n\}$.  Summing over all these indexes and using $\sum_k q_{ik} = 1$ for all $i$, we obtain $v_{i_1} \dotsm \, v_{i_l} = 1$.  It follows that there is a *-homomorphism $\Cstar(\Gamma) \ra A$ sending $u_{g_i}$ to $v_i$ for all $\{1, \dotsc, n\}$.

  \textbf{The $\rS_n$-action on $\Gamma$:}
  Since $g_i^2 = e$, there is a quotient map $\ZZ_2^{*n} \ra \Gamma$ mapping the $i$-th natural generator of $\ZZ_2^{*n}$ to $g_i$.  Let us show that the kernel of this map is $\rS_n$-invariant.  Take any permutation $\sigma \in \rS_n$ and denote by $\chi_\sigma: \cont(\rS_n) \ra \CC$ the associated evaluation map.  We have
  \begin{equation*}
    \Delta(v_i) = \Delta(\sum_j u_{ij}) = \sum_{k,j} u_{ik} \ot u_{kj} = \sum_k u_{ik} \ot v_k
    \eqstop
  \end{equation*}
  So using the quotient map $\psi:A \ra \cont(\rS_n): u_{ij} \ra p_{ij}$, we find
  \begin{equation*}
    (\chi_\sigma \circ \psi \ot \id)(\Delta(v_i)) = \sum_k \chi_\sigma( p_{ik}) \ot v_k = v_{\sigma^{-1}(i)}
    \eqstop
  \end{equation*}
  Assume that $g_{i_1} \dotsm g_{i_l} = e$.  Then $v_{i_1} \dotsm v_{i_l} = 1$ and hence $\Delta(v_{i_1} \dotsm v_{i_l}) = 1 \ot 1$.  We obtain
  \begin{equation*}
    1
    =
    (\chi_{\sigma^{-1}} \circ \psi \ot \id)(\Delta(v_{i_1} \dotsm v_{i_l}))
    =
    v_{\sigma(i_1)} \dotsm v_{\sigma(i_l)}
    \eqstop
  \end{equation*}
  This implies that $g_{\sigma(i_1)} \dotsm \, g_{\sigma(i_l)} = e$.  We have shown that the kernel of $\ZZ_2^{*n} \ra \Gamma$ is invariant under the natural action of $\rS_n$.  So $\rS_n$ acts on $\Gamma$ by permuting its generators.

  \textbf{End of the proof:}
  Since $\rS_n$ acts on $\Gamma$, we can hence consider $\Cstar(\Gamma) \ot \cont(\rS_n)$ with the fundamental corepresentation matrix $w = (u_{g_i}p_{ij}) \in \rM_n(\Cstar(\Gamma) \ot \cont(\rS_n))$ as described in Section~\ref{sec:semi-direct-products}.  Now consider the map $\rho: \Cstar(\Gamma) \ot \cont(\rS_n) \ra A$ defined by $\rho(u_{g_i}) = v_i$ and $\rho(p_{ij}) = q_{ij}$.  Then $\rho(u_{g_i}p_{ij}) = u_{ij}$ saying that $\rho$ is a morphism of CMQGs.  We prove that $(\pi \ot \psi) \circ \Delta: A \ra \Cstar(\Gamma) \ot \cont(\rS_n)$ is the inverse of $\rho$.  Indeed, it suffices to note that
  \begin{equation*}
    (\pi \ot \psi) \circ \Delta (u_{ij})
    =
    (\pi \ot \psi)(\sum_k u_{ik} \ot u_{kj})
    =
    \sum_k u_{g_i} \delta_{i,k} \ot p_{kj}
    =
    u_{g_i} \ot p_{ij}
    \eqstop
  \end{equation*}
  We proved that $A \cong \Cstar(\Gamma) \Join \cont(\rS_n)$ as CMQGs, which concludes the proof. 
\end{proof}

Before we end this section, let us mention the following proposition showing the particular relevance of diagonal subgroups for CMQGs $(A,u)$ for which $u_{ij}^2$ are central projections in $A$.  This makes the assumptions of Theorem \ref{thm:classification-of-qsubgroups} very natural.
\begin{proposition}
  \label{prop:diagonal-subgroups-restriction-to-group-theoretical-case}
  Let $(A,u)$ be an orthogonal CMQG and let $(B, v)$ be the quotient of $A$ by the relations
  \begin{equation*}
    u_{ij}^2 \text{ is a central projection for all } i,j
  \end{equation*}
  Then $\mathrm{diag}(A,u) = \mathrm{diag}(B,v)$.
\end{proposition}
\begin{proof}
  Denote by $\Gamma_A$ the diagonal subgroup of $A$ and by $g_i$, $i \in \{1, \dotsc, n\}$ its generators.  Then the diagonal subgroup of $(B,u)$ is described by its \Cstar-algebra via $\Cstar(\Gamma_B) = \Cstar(\Gamma_A)/ \{u_{g_i}^2 \text{ is a central projection}\}$.  Since $(A, u)$ is orthogonal, the generators of its diagonal subgroup satisfy $g_i^2 = e$ for all $i \in \{1, \dotsc, n\}$.  It follows that $u_{g_i}^2 = 1$ and hence $\Cstar(\Gamma_B) = \Cstar(\Gamma_A)$, which finishes the proof.
\end{proof}

\section{Easy quantum subgroups of $\Cstar(\ZZ_2^{*n}) \Join \cont(\rS_n)$}
\label{sec:classification-partition-qgs}

Recall from the preliminaries that the easy quantum group $A = \cont(H_n^{[\infty]})$ associated with the category $\langle \primarypart \rangle$ is the universal \Cstar-algebra generated by the entries $u_{ij}$, $i,j \in \{1, \dotsc, n\}$ of its fundamental corepresentation subject to the relations that $u_{ij}^2$ are central projections in $A$ for all $i,j \in \{1, \dotsc, n\}$.  Since all easy quantum groups are homogeneous, Theorem \ref{thm:classification-of-qsubgroups} implies that $\cont(H_n^{[\infty]}) \cong \Cstar(\ZZ_2^{*n}) \Join \cont(\rS_n)$.  In this section we achieve a complete classification of easy quantum subgroups of $\Cstar(\ZZ_2^{*n}) \Join \cont(\rS_n)$.  In \cite{raumweber13-combinatorial-approach} we present a completely combinatorial proof for the classification of easy quantum subgroups of $\cont(H_n^{[\infty]})$, but it fails to give a description of the quantum groups.  The basic ideas behind \cite{raumweber13-combinatorial-approach} and the present section are similar, but we make use of Theorem \ref{thm:classification-of-qsubgroups} in order to simplify combinatorial considerations.

Theorem \ref{thm:classification-of-qsubgroups} tells us that we need to investigate diagonal subgroups of easy quantum groups in order to describe them completely.

\subsection{Diagonal subgroups of easy quantum groups}
\label{sec:diagonal-subgroups-of-easy-quantum-groups}

Recall that the diagonal subgroup of $\cont(O_n^+)$ is $\ZZ_2^{*n}$ together with its natural generators.  Hence, all diagonal subgroups of quantum subgroups $\cont(O_n^+) \thra (A,u)$ are quotients of $\ZZ_2^{*n}$.  We want to describe these quotients for easy quantum groups.

Let $a_1, a_2, \dotsc $be the natural generators of $\ZZ_2^{*\infty}$.  Given a partition $p \in \cP(k)$ we say that a labelling $(i_1, \dotsc, i_k)$ of $p$ is compatible, if every block of $p$ is labelled by exactly one index  -- equivalently $\delta_p(i) = 1$.  Take $p \in \cP(k)$ and a compatible labelling $(i_1, \dotsc, i_k)$ of $p$ by indexes in $\NN$, we denote by $w(p,i)$ the word in $\ZZ_2^{*\infty}$ which arises by labelling $p$ with the letters $a_{i_1}, a_{i_2}, \dotsc, a_{i_k}$ from left to right.  We write $\underline{n} = \{1, \dotsc, n\}$ for $n \in \NN^\times$ and $\underline{n} = \NN^\times$ for $n = \infty$.  If $\cC$ is a category of partitions we write $\cC(k) = \cC(k,0)$ for the partitions of length $k$ without lower points.  For $n \in \NN \cup \{\infty\}$, we write
\begin{equation*}
  F_n(\cC) = \{w(p,i) \amid k \in \NN, p \in \cC(k), i \in \underline{n}^k \text{ compatible}\}
\end{equation*}
for the set of all possible words  in $\ZZ_2^{*n}$ arising from $\cC$.  The next lemma shows that $F_n(\cC)$ is always a normal subgroup of $\ZZ_2^{*n}$.  
\begin{lemma}
  \label{lem:FC-is-a-group}
  Let $\cC$ be a category of partitions and $n \in \NN \cup \{\infty\}$.  Then $F_n(\cC)$ is a normal subgroup of $\ZZ_2^{*n}$.
\end{lemma}
\begin{proof}
  We first show that $F_n(\cC)$ is closed under products.  Let $p \in \cC(k),p' \in \cC(k')$ and let $i$,$i'$ be compatible labellings with indexes from $\underline{n}$ of length $k$ and $k'$, respectively.  Then
  \begin{equation*}
    w(p,i)w(p',i') = w(p \ot p', (i_1, \dotsc, i_k, i'_1, \dotsc, i'_{k'}))
    \eqstop
  \end{equation*}
  Next, observe that for $p \in \cC(k)$ with compatible labelling $i$, the inverse of $w(p,i)$ is given by
  \begin{equation*}
    w(p,i)^{-1} = w(p^*, (i_k,i_{k-1}, \dotsc, i_1))
    \eqstop
  \end{equation*}
  It remains to show that $F_n(\cC)$ is normal in $\ZZ_2^{*n}$.  Take $p$ and $i$ as before.  Choose some $i_0 \in \{1, \dotsc, n\}$.  Then
  \begin{equation*}
    \Ad(a_{i_0})(w(p, i)) = w(p',(i_0,i_1, \dotsc, i_k, i_0))
    \eqcomma
  \end{equation*}
  where $p'$ is the partition arising from $p$ by rotating the uttermost right leg of $p \ot \sqcup$ to the left.  This finishes the proof.
\end{proof}

We can now give a description of the diagonal subgroup of an arbitrary easy quantum group.
\begin{lemma}
  \label{lem:presentation-of-diagonal-subgroups}
  Let $\cC$ be a category of partitions and let $G_\cC(n)$ be the easy quantum group associated with $\cC$ whose fundamental corepresentation matrix has size $n \times n$.   Let $\Gamma$ be the diagonal subgroup of $G_\cC(n)$, whose natural generators we denote by $g_1, \dotsc, g_n$.  Then $\Gamma = (\ZZ_2^{*n}) / F_n(\cC)$, where the natural generators of $\ZZ_2^{*n}$ map to the generators $g_1, \dotsc, g_n$.
\end{lemma}
\begin{proof}
  First note that by construction $\cont(G_\cC(n))$ is the universal \Cstar-algebra generated by the entries of a matrix  $(u_{ij})_{1 \leq i,j \leq n}$ satisfying $T_p u^{\ot k} = u^{\ot l} T_p$ for all $p \in \cC(k,l)$.  Since rotating partitions is implemented by repeated composition with $| \ot \dotsm \ot | \ot \sqcup$, one can equivalently describe the relations of $(u_{ij})$ by
$T_p u^{\ot k} = T_p$ for all $p \in \cC(k)$.  Writing this out, we obtain the relations
\begin{equation*}
  \sum_{i_1, \dotsc, i_k} \delta_p(i) u_{i_1j_1}\dotsm u_{i_kj_k}
  =
  \delta_p(j)
  \eqcomma
\end{equation*}
for all partitions $p \in \cC(k)$ and all indexes $j = (j_1, \dotsc, j_k)$.  Passing to the diagonal subgroup of $\Gamma$ of $\cont(G_\cC(n))$ replaces $u_{ij}$ by $\delta_{ij} u_{g_i}$ in the previous relations, where $g_1, \dotsc, g_n$ denote the natural generators of $G$.  Using Proposition \ref{prop:diagonal-subgroups-of-maximal-versions}, we see that $\Cstar(\Gamma)$ is the universal \Cstar-algebra whose generating elements $u_{g_1}, \dotsc, u_{g_n}$ satisfy
\begin{equation*}
  \delta_p(i) u_{g_{i_1}}\dotsm \, u_{g_{i_k}}
  =
  \delta_p(i)
  \eqcomma
\end{equation*}
for all $p \in \cC(k)$ and all indexes $i_1, \dotsc, i_k$.  So $\Gamma$ is the universal group generated by elements $g_1, \dotsc, g_n$ which satisfy $e = g_{i_1} \dotsm g_{i_k}$ for all partitions $p \in \cC(k)$ and all compatible labellings $i = (i_1, \dotsc, i_k)$.   Note that in particular $g_i^2 = e$ for all $i \in \{1, \dotsc, n\}$, since $\sqcup \in \cC$.  Put differently, $\Gamma$ is a quotient of $\ZZ_2^{*n}$ whose relations are exactly given by elements of $F_n(\cC)$.  This finishes the proof.
\end{proof}

Next we precisely describe the possible diagonal subgroups of easy quantum groups.  The key notion for the subsequent classification is strong symmetry of subgroups of $\ZZ_2^{*n}$.
\begin{definition}
  \label{def:strong-symmetric-semigroup}
  Let $n \in \NN \cup \{\infty\}$.  We define the strong symmetric semigroup $\mathrm{sS}_n$ as the semigroup of $\End(\ZZ_2^{*n})$ consisting of identifications of letters.  It contains precisely the endomorphisms $\sigma_\phi:a_k \mapsto a_{\phi(k)}$, $k \in \underline{n}$, where $\phi:\underline{n} \ra \underline{n}$ is any map.

  A \emph{strongly symmetric reflection group} $\Gamma$ is the quotient $\ZZ_2^{*n} \ra \Gamma$ by an $\mathrm{sS}_n$-invariant, normal subgroup together with its natural generators, which are the images of $a_i$ in $\Gamma$.
\end{definition}

The next theorem contains all necessary combinatorial considerations, in order to deduce our main results about easy quantum groups.
\begin{theorem}
  \label{thm:range-of-F}
  For all categories of partitions and all $n \in \NN \cup \{\infty\}$, the subgroup $F_n(\cC) \leq \ZZ_2^{*n}$ is $\mathrm{sS}_n$-invariant.  Vice versa, for every $n \in \NN \cup \{\infty\}$ and every $\mathrm{sS}_n$-invariant, normal subgroup $N \leq \ZZ_2^{*n}$ there is a category of partitions $\cC$ such that $F_n(\cC) = N$.
\end{theorem}
\begin{proof}
  Let $\cC$ be a category of partitions and $n \in \NN \cup \{\infty\}$.  We prove that $F_n(\cC)$ is $\mathrm{sS}_n$-invariant.  So let $p \in \cC(k)$ and $i$ be a compatible labelling of $p$ by elements of $\underline{n}$.  Let $w(p,i) \in F_n(\cC)$ be the associated word in $\ZZ_2^{*n}$ and $\phi:\underline{n} \ra \underline{n}$ be an arbitrary map.  We denote by $\phi_*(i)$ the labelling $(\phi(i_1), \dotsc, \phi(i_k))$.  Then $\sigma_\phi(w(p,i)) = w(p, \phi_*(i))$ showing that $F_n(\cC)$ is $\mathrm{sS}_n$-invariant.

  Fix $n \in \NN \cup \{\infty\}$.  We show that all $\mathrm{sS}_n$-invariant, normal subgroups $N \leq \ZZ_2^{*n}$ arise as $N = F_n(\cC)$ for some category of partitions $\cC$.  Let $N \leq \ZZ_2^{*n}$ be an $\mathrm{sS}_n$-invariant, normal subgroup and define 
  \begin{equation*}
    \cC
    =
    \{p \in \cP \amid p \text{ is a rotation of } \ker(i) \text{ for some } k \in \NN, a_{i_1} \dotsm \, a_{i_k} \in N\}
    \eqstop
  \end{equation*}
  If $p = \ker(i)$ for some $a_{i_1} \dotsm \, a_{i_k} \in N$ and $p' \in \cC(k)$ is a rotation of $p$, then also $p' = \ker(i')$ for some $a_{i'_1} \dotsm \, a_{i'_k} \in N$.  Indeed, it suffices to check this for a rotation of the uttermost left leg of $p$ to the uttermost right of $p$.  If $p'$ arises from $p$ in this way, then $p' = \ker(i_2, \dotsc, i_k, i_1)$ and $a_{i_2} \dotsm \, a_{i_k} a_{i_1} = \Ad (a_{i_1}) (a_{i_1} \dotsm \, a_{i_k}) \in N$ by normality of $N \leq \ZZ_2^{*n}$.  In particular
  \begin{equation}
    \label{eq:upper-row}
    \cC(k)
    =
    \{p \in \cP \amid p = \ker(i) \text{ for some } a_{i_1} \dotsm \, a_{i_k} \in N\}
    \eqstop
  \end{equation}

  We show that $\cC$ is a category of partitions.
  \begin{itemize}
  \item We have $a_1^2 = e \in N$, so $\sqcup = \ker((1,1)) \in \cC$.
  \item It is clear that $\cC$ is closed under rotation.
  \item We show that $\cC$ is closed under involution.  The involution of a rotation of $p \in \cC$ is equal to a rotation of the involution of $p$.  So we have to check that $\ker((i_k, \dotsc, i_1)) \in \cC$ for all $a_{i_1} \dotsm \, a_{i_k} \in N$.  This follows from $a_{i_k} \dotsm \, a_{i_1} = (a_{i_1} \dotsm \, a_{i_k})^{-1}$.
  \item The tensor product of partitions $p,p' \in \cC$ (not necessarily on one row) is a rotation of the tensor product of rotations of $p$ and $p'$ onto one row.  So we have to check that for $a_{i_1} \dotsm \, a_{i_k} \in N$ and $a_{i'_1} \dotsm \, a_{i'_{k'}} \in N$ we have $\ker(i) \ot \ker(i') \in \cC$.  Since $N$ is $\rS_n$-invariant, we may apply a permutation of letters to $a_{i'_1} \dotsm \, a_{i'_{k'}}$ in order to assume that $\{i_1, \dotsc, i_k\} \cap \{i'_1, \dotsc, i'_{k'}\} = \emptyset$.  But then $a_{i_1} \dotsm a_{i_k} a_{i'_1} \dotsm \, a_{i'_{k'}} \in N$ implies $\ker(i) \ot \ker(i') = \ker((i_1, \dotsc, i_k, i'_1, \dotsc, i'_{k'})) \in \cC$.
  \item It remains to show that if $p \in \cC(k,l)$ and $q \in \cC(l,m)$ then also $qp \in \cC(k,m)$.  We may rotate all lower legs of $p$ to its right and likewise for the lower legs of $q$, so as to obtain partition $p' \in \cC(k + l)$ and $q' \in \cC(l + m)$.  Then $qp$ is a rotation of the $l$-fold iterated composition of $p' \ot q'$ with partitions of the form $| \ot \dotsm \ot | \ot \sqcap \ot | \ot \dotsm \ot |$.  So we have to show that the composition of any $p \in \cC(k)$ with a partition of the form $| \ot \dotsm \ot | \ot \sqcap \ot | \ot \dotsm \ot |$ lies in $\cC(k-2)$.  By (\ref{eq:upper-row}), there is $a_{i_1} \dotsm \, a_{i_k} \in N$ such that $p = \ker(i)$.  Denote by $\sqcap_{k,l}$ the partition in $\cC(0,k)$ of the from $| \ot \dotsm \ot | \ot \sqcap \ot | \ot \dotsm \ot |$ such that the first leg of $\sqcap$ is on the $l$-th position of $\sqcap_{k,l}$.  Denote by $\phi: \underline{n} \ra \underline{n}$ the map that satisfies $\phi(i_{l}) = i_{l+1}$ and that fixes all other elements of $\underline{n}$.  By $\mathrm{sS_n}$-invariance of $N$, we have
    \begin{equation*}
      a_{\phi(i_1)} \dotsm \, a_{\phi(i_{l - 1})} a_{\phi(i_{l+2})} \dotsm \, a_{\phi(i_k)}
      =
      \phi_*(a_{i_1} \dotsm \, a_{i_k}) \in N
      \eqstop
    \end{equation*}
    So the composition $p \circ \sqcap_{k,l} = \ker((\phi(i_1), \dotsc, \phi(i_{l-1}), \phi(i_{i+2}), \dotsc, \phi(i_k)))$ lies in $\cC(k-2)$.
  \end{itemize}
  We have shown that $\cC$ is a category of partitions.  Since $N$ is $\rS_n$-invariant, (\ref{eq:upper-row}) shows that $F_n(\cC) = N$, which finishes the proof of the theorem.
\end{proof}

\subsection{Classification of group-theoretical easy quantum groups}
\label{sec:classification-of-group-theoretical-easy-quantum-groups}

We can now combine the previous theorem with the classification of quantum groups in Theorem~\ref{thm:classification-of-qsubgroups} and the classification of diagonal subgroups of easy quantum groups given in Lemma~\ref{lem:presentation-of-diagonal-subgroups}.  In view of Proposition \ref{prop:diagonal-subgroups-restriction-to-group-theoretical-case}, it is natural to consider easy quantum groups, whose categories of partitions contain $\primarypart$.  This yields a complete description of group-theoretical easy quantum groups.
\begin{theorem}
  \label{thm:classification-easy-quantum-groups-bicrossed-product-case}
  Let $(A,u)$ be an easy quantum group associated with the category of partitions $\cC$.  Assume that $\primarypart \in \cC$.  Then $A \cong \Cstar(\Gamma) \Join \cont(\rS_n)$ as compact matrix quantum groups for the strongly symmetric reflection group $\Gamma = \ZZ_2^{*n}/F_n(\cC)$.

Moreover, for every strongly symmetric reflection group $\Gamma$ there is an easy quantum group which is isomorphic with $\Cstar(\Gamma) \Join \cont(\rS_n)$ and whose category of partitions contains $\primarypart$.
\end{theorem}
\begin{proof}
  Since $A$ is an easy quantum group, it is a homogeneous quantum group in its maximal version.  So we can apply Theorem~\ref{thm:classification-of-qsubgroups} showing that $A \cong \Cstar(\Gamma) \Join \cont(G)$, where $\Gamma$ is the diagonal subgroup of $(A,u)$, the group $G$ is some subgroup of $\rS_n$ and the entry $u_{ij}$ of the fundamental corepresentation of $A$ is identified with $u_{g_i}p_{ij}$.  Here $g_1, \dotsc, g_n$ denote the natural generators of $\Gamma$ and $(p_{ij})$ is the fundamental corepresentation of $\cont(G)$ given by the natural embedding $G \hra S_n$.  We have to show that $G = \rS_n$, that $\Gamma \cong (\ZZ_2^{*n})/F_n(\cC)$ and that the latter is a strongly symmetric reflection group.  First note that there is a homomorphism of CMQGs $\pi:A \ra \cont(\rS_n)$ given by $\pi(u_{ij}) = \pi(u_{ij}^2) = p_{ij}$.  It follows that $G = \rS_n$.  Next, Theorem~\ref{thm:range-of-F} shows that $\Gamma = (\ZZ_2^{*n})/F_n(\cC)$ is strongly symmetric.

  Theorem~\ref{thm:range-of-F} also shows that for all strongly symmetric reflection groups on $n$ generators there is a category of partitions $\cC$ such that $\Gamma \cong \ZZ_2^{*n}/F_n(\cC)$.  We can invoke Propositions \ref{prop:description-hyperoctahedral-quantum-groups} and \ref{prop:diagonal-subgroups-restriction-to-group-theoretical-case} so as to assume $\primarypart \in \cC$.  But then the first part of the proof shows that $A_\cC(n) \cong \Cstar(\Gamma) \Join \cont(\rS_n)$.  This finishes the proof.
\end{proof}

The next theorem deduces a complete classification of group-theoretical categories of partitions.
\begin{theorem}
  \label{thm:1-1-correspondence}
  There is a one-to-one correspondence between
  \begin{itemize}
  \item categories of partitions which contain $\primarypart$ and
  \item strongly symmetric reflection groups on countably many generators
  \end{itemize}
  It is given by associating $\ZZ_2^{*\infty}/ F_\infty(\cC)$ with a category of partitions $\cC$.
\end{theorem}
\begin{proof}
  By Theorem \ref{thm:range-of-F} every $\mathrm{sS}_\infty$-invariant, normal subgroup $N \leq \ZZ_2^{*\infty}$ arises as $F_\infty(\cC)$ for some category of partitions $\cC$.  Since $F_\infty(\cC) = F_\infty(\langle \cC, \primarypart \rangle)$ by Proposition \ref{prop:diagonal-subgroups-restriction-to-group-theoretical-case}, we can assume that $\cC$ contains $\primarypart$.  This shows that $\cC \mapsto \ZZ_2^{*\infty}/ F_\infty(\cC)$ maps surjectively from all categories of partitions which contain $\primarypart$ onto strongly symmetric reflection groups on countably many generators.  It remains to show that this map is injective.  
    
  Let $\cC$ be a category of partitions which contains $\primarypart$.  We show that $\cC$ can be recovered from $\ZZ_2^{*\infty}/F_\infty(\cC)$.  Note that $\ZZ_2^{*n}/F_n(\cC)$ is the subgroup of $\ZZ_2^{*\infty}/F_\infty(\cC)$ generated by the first $n$ generators.  Moreover, $\cont(G_\cC(n)) \cong \Cstar(\ZZ_2^{*n}/F_n(\cC)) \Join \cont(\rS_n)$ for all $n \in \NN$ by Theorem \ref{thm:classification-easy-quantum-groups-bicrossed-product-case}.  So Theorem \ref{thm:classification-easy-quantum-groups-by-partitions} implies that $\cC$ can be recovered from $F_\infty(\cC)$. 
\end{proof}

\section{Applications}
\label{sec:Applications}

\subsection{Uncountably many pairwise non-isomorphic easy quantum groups}
\label{sec:uncountably-many-easy-quantum-groups}

In this section we show that there are uncountably many easy quantum groups.  We inject the lattice of varieties of groups into the lattice of hyperoctahedral categories of partitions.  The generators of $\freegrp{\infty}$ are denoted by $x_1, x_2, \dotsc$.
\begin{definition}
  \label{def:variety-of-groups}
  Let $w = w(x_1, \dotsc, x_n) \in \freegrp{\infty}$ be a word in the letters $x_1, \dotsc, x_n$ and let $\Gamma$ be a group.  Then the \emph{identical relation $w$ holds in $\Gamma$} if for all $g_1, \dotsc, g_n \in \Gamma$, we have $w(g_1, \dotsc, g_n) = e$.

  Let $W \subset \freegrp{\infty}$ be any subset of the free group on countably many generators $x_1, x_2, \dotsc$.  The \emph{variety of groups} $\cV(W)$ associated with $W$ is the class of all groups $\Gamma$ such that for all $w \in W$ the identical relation $w$ holds in $\Gamma$.
\end{definition}

Let us state a classical result in the theory of varieties of groups.  A subgroup $N \leq \freegrp{\infty}$ is called \emph{fully characteristic} if $\phi(N) \subset N$ for all $\phi \in \End(\freegrp{\infty})$.
\begin{theorem}[See \cite{neumann37} or Theorem 14.31 in \cite{neumann67}]
\label{thm:varieties-and-fully-characteristic-subgroups}
  There is a lattice anti-isomorphism between varieties of groups and fully characteristic subgroups of $\freegrp{\infty}$ sending a variety of groups to the set of all identical relations that hold in it.
\end{theorem}

Denote by  $E \leq \ZZ_2^{* \infty}$ the group consisting of all words of even length.  We identify $E$ with a free group with basis $x_1 = a_1a_2, x_2 = a_1a_3, \dotsc$.  The following observation is key to this section.

\begin{proposition}
  \label{prop:E-is-fully-characteristic}
  The subgroup $E \leq \ZZ_2^{* \infty}$ is fully characteristic.  In particular, every fully characteristic subgroup of $\freegrp{\infty} \cong E$ is fully characteristic in $\ZZ_2^{* \infty}$.
\end{proposition}
\begin{proof}
  If $w = a_{i_1} a_{i_2} \dotsm a_{i_{2k}} \in E$ and $\phi \in \End(\ZZ_2^{* \infty})$, then $\phi(w) = \phi(a_{i_1}) \dotsm \phi(a_{i_{2k}}) \in E$.  So $E$ is fully characteristic in $\ZZ_2^{* \infty}$.  So every endomorphism of $\ZZ_2^{* \infty}$ restricts to $E$. This shows that a fully characteristic subgroup $N \leq E \cong \freegrp{\infty}$ is also a fully characteristic subgroup in $\ZZ_2^{* \infty}$.  This finishes the proof.
\end{proof}

\begin{remark}
  We do not know whether any fully characteristic subgroup $\ZZ_2^{*\infty}$ that is contained in $E$ and is fully characteristic, is also fully characteristic in $\freegrp{\infty}$.  If this was the case, then hyperoctahedral easy quantum groups $A$ whose entries are subject to the condition that $u_{ij}^2$ is central in $A$, would be classified by non-trivial varieties of groups.
\end{remark}

Combining the previous proposition with the classification of group-theoretical hyperoctahedral categories of partitions, we obtain the following result.
\begin{theorem}
\label{thm:easy-quantum-groups-and-varieties}
  There is a lattice injection from the lattice of non-trivial varieties of groups into the lattice of easy quantum groups.
\end{theorem}
\begin{proof}
  By Theorem \ref{thm:classification-easy-quantum-groups-by-partitions} there is a lattice anti-isomorphism between categories of partitions and easy quantum groups.  Moreover by Theorem \ref{thm:1-1-correspondence}, group-theoretical hyperoctahedral categories of partitions are in one-one correspondence with $\mathrm{sS}_\infty$-invariant, normal subgroups of $\ZZ_2^{*\infty}$ and this correspondence preserves the lattice structure given by inclusion.
  
  By Theorem \ref{thm:varieties-and-fully-characteristic-subgroups}, there is a lattice anti-isomorphism between varieties of groups and fully characteristic subgroups of $\freegrp{\infty}$.

  So it suffices to find an embedding of lattices from fully characteristic subgroups of $\freegrp{\infty}$ into $\mathrm{sS}_\infty$-invariant, normal subgroups of $\ZZ_2^{*\infty}$.  By Proposition \ref{prop:E-is-fully-characteristic}, the embedding $\freegrp{\infty} \cong E \leq \ZZ_2^{*\infty}$ has this property.  This finishes the proof.
\end{proof}

The correspondence from the last theorem allows us to translate classification results for varieties of groups into results on easy quantum groups.  In \cite{banicacurranspeicher09_2}, the question was raised whether or not all easy quantum groups are either classical, free, half-liberated or form part of a multi-parameter family unifying the series of quantum groups $H_n^{(s)}$ and $H_n^{[s]}$.  We can answer this question in the negative.

\begin{theorem}
\label{thm:uncountably-many-easy-quantum-groups}
  There are uncountably many pairwise non-isomorphic easy quantum groups.
\end{theorem}

This follows directly from Theorem \ref{thm:easy-quantum-groups-and-varieties} and the following result by Olshanskii.

\begin{theorem}[See \cite{olshanskii70}]
\label{thm:uncountably-many-varieties}
  The class of varieties of groups has cardinality equal to the continuum.
\end{theorem}

\subsection{There are non-easy quantum groups}
\label{sec:non-partiton-quantum-groups}

It was an open question, whether the formalism of easy quantum groups describes homogeneous orthogonal compact matrix quantum groups, i.e. all intermediate quantum groups $O_n^+ \supset G \supset \rS_n$.  We also answer this question in the negative.

\begin{theorem}
  \label{thm:non-easy-quantum-group}
  For $n \geq 3$ there is an example of an intermediate quantum group $O_n^+ \supset G \supset \rS_n$ such that $G$ is not an easy quantum group.
\end{theorem}
\begin{proof}
  Let $n \geq 3$.  We exhibit an example of a non-easy quantum group $H_n^{[\infty]} \supset G \supset \rS_n$. By Theorem~\ref{thm:classification-of-qsubgroups} it must satisfy $\cont(G) \cong \Cstar(\Gamma) \Join \cont(S_n)$ for some $\rS_n$-invariant quotient $\ZZ_2^{*n} \thra \Gamma$.  So by Theorem~\ref{thm:classification-easy-quantum-groups-bicrossed-product-case} it suffices to exhibit an $\rS_n$-invariant, normal subgroup of $\ZZ_2^{*n}$ that is not $\mathrm{sS}_n$-invariant.   

  Let $\pi: \ZZ_2^{*n} \ra \rS_{n+1}$ be the homomorphism satisfying $\pi(a_i) = (1 i)$.  It is surjective, since $(i j) = (1 i)(1 j)(1 i)$.  If $\phi: \{1, \dotsc, n\} \ra \{1, \dots, n\}$ is any map, then $\pi(a_{\phi(i)}) = (1 \phi(i))$.  So if $w \in \ZZ_2^{*n}$ maps to a product of cycles $\pi(w) = (i_1  \dotsc i_{k_1}) \dotsm \, (i_{k_l-1} \dotsm i_{k_l})$, then $\pi(\sigma_\phi(w)) = (\phi(i_1) \phi(i_2) \dotsc \phi(i_{k_1})) \dotsm \, (\phi(i_{k_l-1}) \dotsm \phi(i_{k_l}))$.  In particular, $\ker \pi$ is $\rS_n$-invariant.  We show that it is not $\mathrm{sS}_n$-invariant.  Indeed, take for $\phi:\{1, \dotsc, n\} \ra \{1, \dotsc, n\}$ the map defined by $\phi(4) = 1$, $\phi(i) = i$ for all $i \neq 4$.  We have $(12)(34)(12)(34) = \id$, but applying $\phi$, we obtain
  \begin{equation*}
    (12)(31)(12)(31)
    =
    (132)
    \neq
    \id
    \eqstop
  \end{equation*}
  So $\ker \pi$ is not $\mathrm{sS}_n$-invariant.  This completes the proof of the theorem.
\end{proof}

\subsection{The hyperoctahedral series}
\label{sec:hyperoctahedral-series}

In \cite{banicacurranspeicher09_2}, the hyperoctahedral series and the higher hyperoctahedral series were defined.  We describe these quantum groups in the context of our classification results.

\begin{definition}
  The \emph{higher hyperoctahedral series} is the sequence of compact matrix quantum groups defined by
  \begin{multline*}
    \cont(H_n^{[s]}) = \Cstar(u_{ij}, 1 \leq i,j \leq n \amid
    u = \overline{u} \text{ unitary, } u_{ij} \text{ partial isometries and } \\
    (u_{ij}u_{kl})^s = (u_{kl}u_{ij})^s \text{ for all } i,j,k,l \in \{1, \dotsc, n\} )
    \eqstop
  \end{multline*}
  Similarly, the \emph{hyperoctahedral series} is defined as
  \begin{multline*}
    \cont(H_n^{(s)}) = \Cstar(u_{ij}, 1 \leq i,j \leq n \amid
    u = \overline{u} \text{ unitary, } u_{ij} \text{ partial isometries, } \\
    (u_{ij}u_{kl})^s = (u_{kl}u_{ij})^s \text{ for all } i,j,k,l \in \{1, \dotsc, n\} \\
    \text{ and } abc = cba \text{ for all } a,b,c \in \{u_{ij} \amid i,j \in \{1, \dotsc, n\})
    \eqstop
  \end{multline*}
\end{definition}

All members of the hyperoctahedral series and the higher hyperoctahedral series are easy quantum groups such that $u_{ij}^2$ are central projections in their underlying \Cstar-algebra.  Hence these quantum groups are in the scope of this article.

\begin{proposition}
  \label{prop:description-hyperocthedral series}
  We have the following isomorphisms of CMQGs.
  \begin{itemize}
  \item $\cont(H_n^{[s]}) \cong \Cstar((\ZZ_2^{*n})/{(a_ia_j)^s = e}) \Join \cont(\rS_n)$.
  \item $\cont(H_n^{(s)}) \cong \Cstar((\ZZ_2^{*n})/{(a_ia_j)^s = e \text{ and } a_ia_ja_k = a_ka_ja_i}) \Join \cont(\rS_n)$.
  \end{itemize}
\end{proposition}
\begin{proof}
  By \cite{banicacurranspeicher09_2}, the higher hyperoctahedral series is associated with the category of partitions generated by $\primarypart$ and
  \setlength{\unitlength}{0.5cm}
  \begin{center}
    \begin{picture}(13,4)
      \put(-1,1.5){$h_s\;=$}
      \put(-0.1,1){\uppartiii{2}{1}{3}{5}}
      \put(5,3){\line(1,0){1}}
      \put(6.75,2.95){$\ldots$}
      \put(-0.1,1){\uppartiii{1}{2}{4}{6}}
      \put(6,2){\line(1,0){1}}
      \put(7.75,1.95){$\ldots$}
      \put(-0.1,1){\uppartii{1}{10}{12}}
      \put(8.65,3){\line(1,0){0.5}}
      \put(-0.1,1){\uppartii{2}{9}{11}}
      \put(9.5,2){\line(1,0){1}}
      \put(12.25,1){.}
    \end{picture}
  \end{center}
  The hyperoctahedral series is associated with the categories $\langle \vierpart, h_s, \halflibpart \rangle$, which also contain $\primarypart$.  So we can apply Lemma \ref{lem:presentation-of-diagonal-subgroups} and Theorem \ref{thm:classification-easy-quantum-groups-bicrossed-product-case} in order to finish the proof.
\end{proof}

The next theorem describes some operator algebraic properties of quantum groups in the hyperoctahedral series.  We refer the reader to Section \ref{sec:operator-algebraic-properties} for the relevant notations.
\begin{theorem}
  \label{thm:properties-hyperoctahedral-series}
  The quantum groups $\cont(H_n^{(s)})$ in the hyperoctahedral series are finite index extensions of $\Cstar(\ZZ_s^{\oplus n-1} \rtimes \rS_n)$, where a permutation $\sigma \in \rS_n$ acts by
  \begin{equation*}
    \sigma(b_i) = 
    \begin{cases}
      b_{\sigma(i)} & \sigma(n) = n \\
      b_{\sigma(n)} & \sigma(i) = n \\
      b_{\sigma(n)}b_{\sigma(i)} & \sigma(n) \neq n \neq \sigma(i)
    \end{cases}
  \end{equation*}
  on the generators $b_1, \dotsc, b_{n-1}$ of $\ZZ_s^{\oplus n-1}$.  In particular, $\wh{H_n^{(s)}}$ is an amenable quantum group, or equivalently $\Linfty(H_n^{(s)})$ is an amenable von Neumann algebra.

  For all $s \geq 3$, $n \geq 3$ the discrete dual quantum groups $\wh{H_n^{[s]}}$ of the higher hyperoctahedral series are not amenable, but weakly amenable and they have the Haagerup property.  Their von Neumann algebras $\Linfty(H_n^{[s]})$ are strongly solid.
\end{theorem}
\begin{proof}
  Let us first consider $H_n^{(s)}$.  By Proposition \ref{prop:description-hyperocthedral series}, the CMQG $\cont(H_n^{(s)})$ is isomorphic with $\Cstar((\ZZ_2^{*n})/{\{(a_ia_j)^s = e \text{ and } a_ia_ja_k = a_ka_ja_i\}}) \Join \cont(\rS_n)$.  Consider the index 2 subgroup $E$ of $\ZZ_2^{*n}/{\{(a_ia_j)^s = e \text{ and } a_ia_ja_k = a_ka_ja_i\}}$ consisting of words of even length.  It is generated by the elements $b_i = a_na_i$, $i \in \{1, \dotsc, n-1\}$.  We can use the relation $a_ia_ja_k = a_ka_ja_i$ in order to obtain $b_ib_j = a_na_ia_na_j = a_na_ja_na_i = b_jb_i$.  Similarly, $b_ib_j = b_jb_i$ implies $a_ia_ja_ka_l = b_i^{-1}b_jb_k^{-1}b_l = b_k^{-1}b_jb_i^{-1}b_l = a_ka_ja_ia_l$ for all $i,j,k,l \in \{1, \dotsc, n\}$.  So $E$ is the universal group generated by $n-1$ commuting elements of order $s$.  Hence $E \cong \ZZ_s^{\oplus n-1}$ by an isomorphism identifying $b_i$, $i \in \{1, \dotsc, n-1\}$ with the natural generators of $\ZZ_s^{\oplus n-1}$.  Take $\sigma \in \rS_n$.  Then $\sigma(b_i) = \sigma(a_na_i) = a_{\sigma(n)}a_{\sigma(i)} = a_na_{\sigma(n)} a_n a_{\sigma(i)}$.  So $\sigma$ acts on $E$ via
  \begin{equation*}
    \sigma(b_i) = 
    \begin{cases}
      b_{\sigma(i)} & \sigma(n) = n \\
      b_{\sigma(n)} & \sigma(i) = n \\
      b_{\sigma(n)}b_{\sigma(i)} & \sigma(n) \neq n \neq \sigma(i)
      \eqstop
    \end{cases}
  \end{equation*}
  In particular, $\rS_n$ leaves $E$ invariant, so that $\Cstar(E) \Join \cont(\rS_n) \subset \cont(H_n^{(s)})$ is a compact quantum group.  Note that since the Pontryagin dual of $\ZZ_s$ is isomorphic with $\ZZ_s$, we have $\Cstar(E) \Join \cont(\rS_n) \cong {\cont(\ZZ_s^{\oplus n-1} \rtimes \rS_n)}$.  Since $E \leq (\ZZ_2^{*n})/{\{(a_ia_j)^s = e \text{ and } a_ia_ja_k = a_ka_ja_i\}}$ has index 2, $\Linfty(H_n^{(s)})$ is finitely generated as an $\Linfty(\ZZ_s^{\oplus n-1} \rtimes S_n)$-module.  So the inclusion $\Linfty(\ZZ_s^{\oplus n-1} \rtimes \rS_n) \subset \Linfty(H_n^{(s)})$ has finite index, implying that $\Linfty(H_n^{(s)})$ is an amenable von Neumann algebra.  By Theorem \ref{thm:amenability-for-quantum-groups} this implies amenability of the discrete dual $\wh{H_n^{(s)}}$.

Let us now consider $\cont(H_n^{[s]})$.  By Proposition \ref{prop:description-hyperocthedral series} it is isomorphic with ${\Cstar(\ZZ_2^{*n}/{\{(a_ia_j)^s = e \}}) \Join \cont(\rS_n)}$.  Denote by $E \leq \ZZ_2^{*n}/{\{(a_ia_j)^s = e \}}$ the subgroup of words of even length.  It is generated by elements $b_i = a_na_i$, which only satisfy the relations $b_i^s = e$.  Hence $E \cong \ZZ_s^{*n-1}$.  Since $u_{a_i} = \sum_j u_{ij} \in  \mathrm{Pol}(\cont(H_n^{[s]}) \subset \Linfty(H_n^{[s]})$, we have $\rL(E) \subset \Linfty(H_n^{[s]})$.  Note that $E$ is not amenable, because $s,n \geq 3$.  It follows that $\Linfty(H_n^{[s]})$ is not amenable as a von Neumann algebra.  By Theorem \ref{thm:amenability-for-quantum-groups} this implies non-amenability of $\wh{H_n^{(s)}}$.  Moreover, $\rL(E) \subset \Linfty(H_n^{[s]})$ has finite index.  Since $E \cong \ZZ_s^{*n-1}$ has the CMAP by Theorem \ref{thm:free-products-preserve-CMAP} and has the Haagerup property by Theorem \ref{thm:free-products-preserve-H}, it follows that also $\Linfty(H_n^{[s]})$ has the \Wstar-CCAP and the Haagerup property. Finally note that $E$ is a free product of hyperbolic groups, so it is hyperbolic itself.  Furthermore every non-trivial conjugacy class in $E \cong \ZZ_s^{* n-1}$ is infinite, because $s,n \geq 3$.  By \cite{chifansinclair13}, it follows that $\rL(E)$ is a strongly solid von Neumann algebra.  Since $\Linfty(H_n^{[s]})$ contains $\rL(E)$ as a finite index von Neumann subalgebra, \cite[Proposition 5.2]{houdayer10-non-free-gropup} implies that also $\Linfty(H_n^{[s]})$ is strongly solid.  This finishes the proof.
\end{proof}

\subsection{De Finetti theorems for easy quantum subgroups of $\Cstar(\ZZ_2^{*n}) \Join \cont(\rS_n)$}
\label{sec:de-finetti-theorems}

After the work of K{\"o}stler and Speicher in \cite{koestlerspeicher09-de-finetti}, de Finetti theorems became a central piece of the theory of easy quantum groups.  In this section we present a unified de Finetti theorem for all easy quantum subgroups of $\Cstar(\ZZ_2^{*n}) \Join \cont(\rS_n)$.  Unsurprisingly, it needs strong assumptions to yield the desired equivalence between invariance of the distribution of non-commutative random variables $x_1, x_2, \dotsc$ under the natural action of a series of easy quantum groups and independence properties of this distribution.  However, the de Finetti theorem for $H_n^* = H_n^{(\infty)}$ as it is known from \cite{banicacurranspeicher12-finetti} takes its usual form, only demanding a commutation relation between the random variables $x_1, x_2, \dotsc$ in question.  Similarly, one can formulate a simple de Finetti theorem for $H_n^{[\infty]}$.  We describe these two especially interesting cases in Corollary \ref{cor:de-finetti-theorem-for-hyperoctahedral-series} justifying the assumptions of Theorem \ref{thm:de-finitti}.

In the next theorem we use the following notation.  If $w \in \NN^{*n}$ is any word in $n$ letters, then we denote by $e_i(w)$ the exponent of the $i$-th letter in $w$.
\begin{theorem}
  \label{thm:de-finitti}
  Let $\Gamma$ be a non-trivial strongly symmetric reflection group on countably many generators and $K = \ker( \NN^{*\infty} \ra \Gamma) \leq \NN^{*\infty}$ the $\mathrm{sS}_\infty$-invariant subsemigroup of $\NN^{*\infty}$ associated with it.  For $n \in \NN$ denote by $\Gamma_n$ the strongly symmetric reflection groups generated by the first $n$ generators of $\Gamma$.  Let $x_1, x_2, \dotsc$ be a sequence of non-commutative, self-adjoint random variables in a \Wstar-probability space $(M,\phi)$.  Then there is a von Neumann subalgebra $B \subset M$ with $\phi$-preserving expectation $\rE:M \ra B$ such that if $x_1, x_2, \dotsc$ satisfy
  \begin{itemize}
  \item $x_i^2x_j = x_jx_i^2$ for all $i,j \in \NN$ and
  \item $\phi(w(x_1, \dotsc, x_n)) = \phi(\rE(x_1^{e_1(w)}) \dotsm \rE(x_n^{e_n(w)}))$ for all $n \in \NN$ and all words $w \in K$ on $n$ letters,
  \end{itemize}
  then the following are equivalent.
  \begin{enumerate}
  \item $x_1^2, x_2^2, \dotsc, $ are identically distributed and independent with respect to $\rE$ and $\phi(w(x_1, \dotsc, x_n)) = 0$ for all $n \in \NN$ and all $w \in \NN^{*\infty} \setminus K$ in $n$ letters.
  \item For all $n \in \NN$, the distribution of $x_1, \dotsc, x_n$ is invariant under the coaction of the easy quantum group ${\Cstar(\Gamma_n) \Join \cont(\rS_n)}$.
  \end{enumerate}
\end{theorem}
\begin{proof}
We may assume that $M = \Wstar(x_1, x_2, \dotsc)$.  Let $B = \bigcap_{n \geq 1} \Wstar(x_n^2, x_{n+1}^2, \dotsc)$ and note that $B$ lies in the centre of $M$.  In particular, there is a $\phi$-preserving normal conditional expectation $\rE:M \ra B$.

  Consider the coaction $\alpha_n$ of $\Cstar(\Gamma_n) \Join \cont(\rS_n)$ on $\CC\langle X_1, X_2, \dotsc \rangle$ given by 
  \begin{equation*}
    \alpha_n(X_i) = \sum_j X_j \ot u_{ji} = \sum_j X_j \ot u_{g_j} p_{ji}
  \end{equation*}
  for $i \leq n$ and $\alpha(X_i) = X_i \ot 1$ for $i > n$.
  Denote by $\psi$ the state on $\CC\langle X_1, X_2, \dotsc  \rangle$ given by $\psi(w(X_1, \dotsc, X_n)) = \phi(w(x_1, \dotsc, x_n))$ for all $w \in\NN^{*n}$ and all $n \in \NN$.
  Then the distribution of $x_1, x_2, \dotsc$ is invariant under the coaction of $\Cstar(\Gamma_n) \Join \cont(\rS_n)$ if and only if for all indexes $i_1, \dotsc, i_l \in \{1, \dotsc, n\}$ we have
  \begin{equation}
    \label{eq:invarance-start}
    \psi(X_{i_1} \dotsm \, X_{i_l})
    =
    \sum_{j_1, \dotsc, j_l \in \{1, \dotsc, n\}} \psi(X_{j_1} \dotsm \, X_{j_l}) u_{g_{j_1} \dotsm g_{j_l}} p_{j_1i_1} \dotsm \, p_{j_l i_l}
    \eqstop
  \end{equation}
  We are going to manipulate the right hand side of this equation.  First note that we have ${p_{j_1i_1} \dotsm \, p_{j_l i_l} \neq 0}$ if and only if there is $\sigma \in \rS_n$ such that $j_k = \sigma(i_k)$ for all $k \in \{1, \dotsc, l\}$.  Moreover, if $\sigma|_{\{i_1, \dotsc, i_l\}} = \rho|_{\{i_1, \dotsc, i_l\}}$ for $\sigma, \rho \in \rS_n$, then 
  \begin{equation*}
    \psi(X_{\sigma(i_1)} \dotsm \, X_{\sigma(i_l)}) u_{g_{\sigma(i_1)} \dotsm g_{\sigma(i_l)}} p_{\sigma(i_1)i_1} \dotsm \, p_{\sigma(i_l) i_l}
    =
    \psi(X_{\rho(i_1)} \dotsm \, X_{\rho(i_l)}) u_{g_{\rho(i_1)} \dotsm g_{\rho(i_l)}} p_{\rho(i_1)i_1} \dotsm \, p_{\rho(i_l) i_l}
    \eqstop
  \end{equation*}
  Let $L$ be the cardinality of $\{i_1, \dotsc, i_l\}$.  Then
  \begin{multline}
    \label{eq:reformulate-1}
    \sum_{j_1, \dotsc, j_l \in \{1, \dotsc, n\}} \psi(X_{j_1} \dotsm \, X_{j_l}) u_{g_{j_1} \dotsm g_{j_l}} p_{j_1i_1} \dotsm \, p_{j_l i_l} \\
    =
    \frac{1}{(n-L)!}
    \sum_{\sigma \in \rS_n} \psi(X_{\sigma(i_1)} \dotsm \, X_{\sigma(i_l)}) u_{g_{\sigma(i_1)} \dotsm g_{\sigma(i_l)}} p_{\sigma(i_1)i_1} \dotsm \, p_{\sigma(i_l) i_l}
    \eqstop
  \end{multline}
  Let $i_{l + 1}, \dotsc, i_{l+l'}$ be an enumeration of $\{1, \dotsc, n\} \setminus \{i_1, \dotsc, i_l\}$.  Using $p_{j'i}p_{ji} = 0$ for all ${i, j , j' \in \{1, \dotsc, n\}}$, $j \neq j'$, we see that
  \begin{equation}
    \label{eq:reformulate-2}
    \begin{aligned}
    & \quad
    \sum_{\sigma \in \rS_n} \psi(X_{\sigma(i_1)} \dotsm \, X_{\sigma(i_l)}) u_{g_{\sigma(i_1)} \dotsm g_{\sigma(i_l)}} p_{\sigma(i_1)i_1} \dotsm \, p_{\sigma(i_l) i_l} \\
    & =
    \bigl (\sum_{\sigma \in \rS_n} \psi(X_{\sigma(i_1)} \dotsm \, X_{\sigma(i_l)}) u_{g_{\sigma(i_1)} \dotsm g_{\sigma(i_l)}} p_{\sigma(i_1)i_1} \dotsm \, p_{\sigma(i_l) i_l} \bigr ) \bigl (\sum_{j_{l+1}, \dotsc, j_{l + l'} \in \{1, \dotsc, n\}} p_{j_{l+1} i_{l+1}} \dotsm \, p_{j_{l+l'} i_{l+l'}} \bigr ) \\
    & =
    (n - L)!  \sum_{\sigma \in \rS_n} \psi(X_{\sigma(i_1)} \dotsm \, X_{\sigma(i_l)}) u_{g_{\sigma(i_1)} \dotsm g_{\sigma(i_l)}} p_{\sigma(i_1)i_1} \dotsm \, p_{\sigma(i_l) i_l} p_{\sigma(i_{l+1})i_{l+1}} \dotsc p_{\sigma(i_{l+l'}) i_{l+l'}}
    \eqstop
    \end{aligned}
  \end{equation}
  We can now use $p_{ji}^2 = p_{ji}$ for all $i,j \in \{1, \dotsc, n\}$ with the equations (\ref{eq:reformulate-1}) and (\ref{eq:reformulate-2}) in order to reformulate the right hand side of (\ref{eq:invarance-start}) and obtain
  \begin{equation*}
    \sum_{j_1, \dotsc, j_l \in \{1, \dotsc, n\}} \psi(X_{j_1} \dotsm \, X_{j_l}) u_{g_{j_1} \dotsm g_{j_l}} p_{j_1i_1} \dotsm \, p_{j_l i_l}
    =
    \sum_{\sigma \in \rS_n} \psi(X_{\sigma(i_1)} \dotsm \, X_{\sigma(i_l)}) u_{g_{\sigma(i_1)} \dotsm g_{\sigma(i_l)}} p_{\sigma(1) 1} \dotsm \, p_{\sigma(n) n}
    \eqstop
  \end{equation*}

  Using this equation, we see that the distribution of $x_1, x_2, \dotsc$ is invariant under the action of $\Cstar(\Gamma_n) \Join \cont(\rS_n)$ if and only if
  \begin{equation}
    \label{eq:invariance}
    \psi(w(X_1, \dotsc, X_n))
    =
    \sum_{\sigma \in \rS_n} \psi(w(X_{\sigma(1)}, \dotsc, X_{\sigma(n)})) u_{w(g_{\sigma(1)}, \dotsc, g_{\sigma(n)}} p_{\sigma(1) 1} \dotsm \, p_{\sigma(n) n}
    \eqcomma
  \end{equation}
  for all words $w \in \NN^{*n}$.

  Assume (i) and take $n \in \NN$.  If $w \in \NN^{*n} \setminus K$ is a word on the first $n$ letters $X_1, \dotsc, X_n$, then $\psi(w(X_1, \dotsc, X_n)) = 0$.  Furthermore, for every $\sigma \in \rS_n$, there is $w' \in \NN^{*n}$ such that $w(X_{\sigma(1)}, \dotsc, X_{\sigma(n)}) = w'(X_1, \dotsc, X_n)$.  Since $w'$ arises from $w$ by permuting its letters according to $\sigma$, we see that $w' \in \NN^{*n} \setminus K$.  Hence we have $\psi(w(X_{\sigma(1)}, \dotsc, X_{\sigma(n)})) = 0$.  Filling in this information into equation (\ref{eq:invariance}), we see that its left- and right-hand side are equal to zero for $w \in \NN^{*n} \setminus K$.

  There is a natural $\Gamma_n$-grading of $\CC \langle X_1, X_2, \dotsc \rangle$ which grades $w(X_1,\dotsc, X_n)$ by $w(g_1, \dotsc, g_n)$ for all ${w \in \NN^{*n}}$ and which grades $X_k$ by $e$ for $k \geq n+1$.  We have already proven that $\psi(w(X_1, \dotsc, X_n)) = 0$ for all $w \in \NN^{*n} \setminus K$.  This is equivalent to the fact that $\psi$ respects the $\Gamma_n$-grading of $\CC\langle X_1, X_2, \dotsc \rangle$.  In particular, purely graded subspaces of $\CC \langle X_1, X_2, \dotsc, \rangle$ are pairwise orthogonal with respect to $\phi$.  So the GNS-construction gives rise to a well defined $\Gamma_n$-grading of $M$.  Every element in $B$ is purely $e$-graded, because $B$ is a subalgebra of $\Wstar(x_1^2, x_2^2, \dotsc, )$.  It follows that $\rE(w(x_1, \dotsc, x_n)) = 0$ for all $w \in \NN^{*\infty} \setminus K$.  In particular, the fact that $\Gamma$ is not trivial implies that all $x_i$ have even distribution with respect to~$\rE$.  Combining this with the fact that $x_1^2, x_2^2, \dotsc,$ have an identical distribution with respect to~$\rE$, it follows that $x_1, x_2, \dotsc $ are identically distributed with respect to~$\rE$.

  Now take $w \in \NN^{*n} \cap K$.  As before, we see that for $\sigma \in \rS_n$, the word $w' \in \NN^{*n}$ arising from a permutation of letters of $w$ according $\sigma$ satisfies $w(X_{\sigma(1)}, \dotsc, X_{\sigma(n)}) = w'(X_1, \dotsc, X_n)$.  Since $K$ is invariant under permutation of letters, we obtain $w' \in K$ and hence $w(g_{\sigma(1)}, \dotsc, g_{\sigma(n)}) = e$.  This shows
  \begin{multline}
    \label{eq:no-grading}
    \sum_{\sigma \in \rS_n} \psi(w(X_{\sigma(1)}, \dotsc, X_{\sigma(n)})) u_{w(g_{\sigma(1)}, \dotsc, g_{\sigma(n)})} p_{\sigma(1) 1} \dotsm \, p_{\sigma(n) n} \\
    =
    \sum_{\sigma \in \rS_n} \psi(w(X_{\sigma(1)}, \dotsc, X_{\sigma(n)})) p_{\sigma(1) 1} \dotsm \, p_{\sigma(n) n}
    \eqstop
  \end{multline}
  For all $\sigma \in \rS_n$, we have
  \begin{align*}
    \psi(w(X_1, \dotsc, X_n))
    & =
    \phi(w(x_1, \dotsc, w_n)) \\
    & = 
    \phi(\rE(x_1^{e_1(w)})  \dotsm \, \rE(x_n^{e_n(w)}))
    & \text{(assumption on $x_1, \dotsc, x_n$)} \\
    & = 
    \phi( \rE(x_{\sigma(1)}^{e_1(w)})  \dotsm \, \rE(x_{\sigma(n)}^{e_n(w)}))
    & \text{($x_1, \dotsc, x_n$ identically distributed wrt $\rE$}) \\
    & = \psi(w(X_{\sigma(1)}, \dotsc, X_{\sigma(n)}))
    \eqstop
  \end{align*}
  Using this formula and equation (\ref{eq:no-grading}), we obtain
  \begin{align*}
    & \quad
    \sum_{\sigma \in \rS_n} \psi(w(X_{\sigma(1)}, \dotsc, X_{\sigma(n)})) u_{w(g_{\sigma(1)}, \dotsc, g_{\sigma(n)})} p_{\sigma(1) 1} \dotsm \, p_{\sigma(n) n} \\
    & =
    \psi(w(X_1, \dotsc, X_n))    \sum_{\sigma \in \rS_n}  p_{\sigma(1) 1} \dotsm \, p_{\sigma(n) n} \\
    & =
    \psi(w(X_1, \dotsc, X_n))
    \eqstop
  \end{align*}
  This shows that $\psi$ is invariant under the action of $\Cstar(\Gamma_n) \Join \cont(S_n)$.

  Assume (ii). For all $n \in \NN$ and all $w \in \NN^{*n}$, we have
  \begin{equation}
    \label{eq:invariance-2}
    \psi(w(X_1, \dotsc, X_n))
    =
    \sum_{\sigma \in \rS_n} \psi(w(X_{\sigma(1)}, \dotsc, X_{\sigma(n)})) u_{w(g_{\sigma(1)}, \dotsc, g_{\sigma(n)})} p_{\sigma(1) 1} \dotsm \, p_{\sigma(n) n}
    \eqstop
  \end{equation}
  Let $w \in \NN^{*n} \setminus K$.  For $\sigma \in \rS_n$, we see as before that $w(g_{\sigma(1)}, \dotsc, g_{\sigma(n)}) \neq e$.  So (\ref{eq:invariance-2}) can only be true if $\psi(w(X_1, \dotsc, X_n)) = 0$.
  Since $g_k^2 = e$ for all $k \in \{1, \dotsc, n\}$,  we have $w(g_1^2, \dotsc, g_n^2) = e$ for all $n \in \NN^{*n}$.  Hence equation (\ref{eq:invariance-2}) for words of length $2n$ implies
  \begin{equation*}
    \psi(w(X_1^2, \dotsc, X_n^2))
    =
    \sum_{\sigma \in \rS_n} \psi(w(X_{\sigma(1)}^2, \dotsc, X_{\sigma(n)}^2)) p_{\sigma(1) 1} \dotsm \, p_{\sigma(n) n}
    \eqstop
  \end{equation*}
  So the same formula is true if we replace $\psi$ by $\phi$ on both sides.  Recall that $B = \bigcap_{n \geq 1} \Wstar(x_n^2, x_{n+1}^2, \dotsc)$ admits the $\phi$-preserving conditional expectation $\rE:M \ra B$.  Since the $x_1^2, x_n^2, \dotsc$ is a commuting family of random variables, the classical de Finetti theorem implies that the distribution of $x_1^2, x_2^2, \dotsc$ is independent and identically distributed with respect to $\rE$.  So we showed (i), which completes the proof of the theorem.
\end{proof}

\begin{remark}
  We want to remark that the previous theorem confirms the special role of easy quantum groups as correct symmetries of non-commutative distributions.  Indeed,  it is the $\mathrm{sS}_n$-invariance of the kernel $\ker(\NN^{*n} \ra \Gamma)$ for a strongly symmetric reflection group $\Gamma$, which allows for substitution of letters and hence turns the statement $\phi(w(x_1, \dotsc, x_n)) = \phi( \rE(x_1)^{e_1(w)} \dotsm \rE(x_n)^{e_n(w)})$ into a reasonable condition.
\end{remark}

We apply Theorem \ref{thm:de-finitti} to the easy quantum groups $H_n^{[\infty]}$ and $H_n^*$ described in Section \ref{sec:hyperoctahedral-series}.  In particular, we recover the de Finetti theorem for $H_n^*$ first proved in \cite{banicacurranspeicher12-finetti}.  We take over the notion of being ``balanced'' from their work.  A word $w \in \NN^{*n}$ is balanced if and only if $w \in \ker(\NN^{*n} \ra \ZZ_2^{*n}/{\{a_ia_ja_k = a_k a_j a_i\}})$.
\begin{corollary}
  \label{cor:de-finetti-theorem-for-hyperoctahedral-series}
  Let $x_1, x_2, \dotsc $ be a sequence of non-commutative self-adjoint random variables in a \Wstar-probability space $(M, \phi)$.  Then there is a von Neumann subalgebra $B \subset M$ and a $\phi$-preserving conditional expectation $\rE:M \ra B$ such that the following statements hold.
  \begin{enumerate}
  \item If $x_i^2 x_j = x_j x_i^2$ for all $i,j \in \{1, \dotsc, n\}$, then the following are equivalent.
    \begin{itemize}
    \item $x_1^2, x_2^2, \dotsc $ is identically distributed and independent with respect to $\rE$ and $\phi(w(x_1, \dotsc, x_n)) = 0$ for all words $w \in \ker (\NN^{*n} \ra \ZZ_2^{*n})$.
    \item For all $n \in \NN$, the distribution of $x_1, \dotsc, x_n$ is invariant under $\cont(H_n^{[\infty]})$.
    \end{itemize}
  
  \item If $x_ix_jx_k = x_kx_jx_i$ for all $i,j,k \in \{1, \dotsc, n\}$, then the following are equivalent
    \begin{itemize}
    \item $x_1^2, x_2^2, \dotsc$ are identically distributed and independent with respect to $\rE$ and $\phi(w(x_1, \dotsc, x_n)) = 0$ for all non-balanced words $w \in \NN^{*n}$.
    \item For all $n \in \NN$, the distribution of $x_1, \dotsc, x_n$ is invariant under $\cont(H_n^*)$.
    \end{itemize}
  \end{enumerate}
\end{corollary}
\begin{proof}
We may assume that $M = \Wstar(x_1, x_2, \dotsc)$.  Let $B = \bigcap_{n \geq 1} \Wstar(x_n^2, x_{n+1}^2, \dotsc)$ and note that $B$ lies in the centre of $M$.  In particular, there is a $\phi$-preserving conditional expectation $\rE:M \ra B$.

  We are going to apply Theorem \ref{thm:de-finitti} to (i) and (ii).  Before starting to consider these cases one by one, let us observe the following facts.  The invariance of the $\rE$-distribution of $x_1, x_2, \dotsc$ by any of the quantum groups $H_n^{[\infty]}$ or $H_n^*$, implies that their distribution is invariant under the permutation groups $\rS_n$.  In particular, as the random variables $x_1^2, x_2^2, \dotsc$ are pairwise commuting, they are independent and identically distributed with respect to $\rE$.  Moreover, as in Theorem \ref{thm:de-finitti}, it follows that the $\rE$-distribution of $x_i$ is even for all $i$.  So $x_1, x_2, \dotsc$ are identically distributed with respect to $\rE$.

  Let us start to prove (i).  Since $x_i^2x_j = x_jx_i^2$ for all $i, j \in \{1, \dotsc, n\}$ and $K = \ker(\NN^{*n} \ra \ZZ_2^{*n})$ is the smallest subsemigroup of $\NN^{*n}$ which is invariant under $x \mapsto a_ixa_i$, we obtain inductively that $w(x_1, \dotsc, x_n) = x_1^{e_1(w)} \dotsm x_n^{e_n(w)}$ for all $w \in K$.  Furthermore, $e_1(w), \dotsc, e_n(w) \in 2\NN$.  Independence of $x_1^2, \dotsc, x_n^2$ with respect to $\rE$ implies that
  \begin{equation*}
    \phi(w(x_1, \dotsc, x_n))
    = 
    \phi(x_1^{e_1(w)} \dotsm x_n^{e_n(w)})
    =
    (\phi \circ \rE)(x_1^{e_1(w)} \dotsm x_n^{e_n(w)})
    =
    \phi(\rE(x_1^{e_1(w)}) \dotsm \rE(x_n^{e_n(w)}))
    \eqcomma
  \end{equation*}
  for all $w \in K$ on $n$ letters.  So we can apply Theorem \ref{thm:de-finitti} in order to finish the proof of (i).

  In order to apply Theorem \ref{thm:de-finitti} to the situation in (ii), we need to check that $\phi(w(x_1, \dotsc, x_n)) = \phi(\rE(x_1^{e_1(w)}) \dotsm \rE(x_n^{e_n(w)}))$ for all balanced words $w \in \NN^{*n}$.  If $w \in \NN^{*n}$ is a balanced word, then $x_ix_jx_k = x_kx_jx_i$ for all $i,j,k$ implies $w(x_1, \dotsc, x_n) = x_1^{e_1(w)} \dotsm x_n^{e_n(w)}$.  So
  \begin{equation*}
    \phi(w(x_1, \dotsc, x_n))
    =
    \phi(x_1^{e_1(w)} \dotsm x_n^{e_n(w)})
    =
    (\phi \circ \rE)(x_1^{e_1(w)} \dotsm x_n^{e_n(w)}))
    =
    \phi(\rE(x_1^{e_1(w)}) \dotsm \rE(x_n^{e_n(w)}))
    \eqcomma
  \end{equation*}
  where the last equality follows from the fact that $e_i(w) \in 2\NN$ for all $i \in \{1, \dotsc, n\}$ and independence of $x_1^2, \dotsc, x_n^2$ with respect to $\rE$.  So we can indeed apply Theorem \ref{thm:de-finitti}.  This finishes the proof.
 
\end{proof}


\bibliographystyle{mybibtexstyle}
\bibliography{easy-quantum-groups-classification}

\newcommand{\etalchar}[1]{$^{#1}$}
\begin{thebibliography}{DFSW13}\setlength{\itemsep}{-1mm}\setlength{\parsep}{0mm}\small

\bibitem[Ban96]{banica96}
T. Banica.
\newblock {The representation theory of the free compact quantum group
  $\mathrm{O}(n)$. (Th{\'e}orie des repr{\'e}sentations du groupe quantique
  compact libre $\mathrm{O}(n)$)}.
\newblock {\em C. R. Acad. Sci. Paris S{\'e}r. I Math.} \textbf{322} (3),
  241--244, 1996.

\bibitem[Ban99]{banica99}
T. Banica.
\newblock {Symmetries of a generic coaction.}
\newblock {\em Math. Ann.} \textbf{314} (4), 763--780, 1999.

\bibitem[BBC07]{banicabichoncollins07}
T. Banica, J. Bichon, and B. Collins.
\newblock {The hyperoctahedral quantum group.}
\newblock {\em J. Ramanujan Math. Soc.} \textbf{22} (4), 345--384, 2007.

\bibitem[BCS10]{banicacurranspeicher09_2}
T. Banica, S. Curran, and R. Speicher.
\newblock {Classification results for easy quantum groups.}
\newblock {\em Pac. J. Math.} \textbf{247} (1), 1--26, 2010.

\bibitem[BCS12]{banicacurranspeicher12-finetti}
T. Banica, S. Curran, and R. Speicher.
\newblock {De Finetti theorems for easy quantum groups.}
\newblock {\em Ann. Probab.} \textbf{40} (1), 401--435, 2012.

\bibitem[BO08]{brownozawa08}
N.~P. Brown and N. Ozawa.
\newblock {\em {C$^*$-algebras and finite-dimensional approximations.}},
  volume~88 of {\em Graduate Studies in Mathematics}.
\newblock Providence, RI: American Mathematical Society, 2008.

\bibitem[BS09]{banicaspeicher09}
T. Banica and R. Speicher.
\newblock {Liberation of orthogonal Lie groups.}
\newblock {\em Adv. Math.} \textbf{222} (4), 1461--1501, 2009.

\bibitem[CCJ{\etalchar{+}}01]{cherixcowlingjolissaintjulgvalette01}
P.-A. Cherix, M. Cowling, P. Jolissaint, P. Julg, and A. Valette.
\newblock {\em {Groups with the Haagerup property. Gromov's a-T-menability}},
  volume 197 of {\em Progress in Mathematics}.
\newblock {Basel: Birkh{\"a}user Verlag}, 2001.

\bibitem[Cho83]{choda83}
M. Choda.
\newblock {Group factors of the Haagerup type.}
\newblock {\em Proc. Japan Acad. Ser. A Math. Sci.} \textbf{59} (5), 174--177,
  1983.

\bibitem[Con76]{connes76}
A. Connes.
\newblock {Classification of injective factors. Cases II$_1$, II$_\infty$,
  III$_\lambda$, $\lambda\neq 1$.}
\newblock {\em Ann. Math. (2)} \textbf{104}, 73--115, 1976.

\bibitem[Con78]{connes78}
A. Connes.
\newblock {On the cohomology of operator algebras.}
\newblock {\em J. Funct. Anal.} \textbf{28} (2), 248–253, 1978.

\bibitem[CS13]{chifansinclair13}
I. Chifan and T. Sinclair.
\newblock {On the structural theory of II$_1$ factors of negatively curved
  groups.}
\newblock {\em Ann. Sci. {\'E}c. Norm. Sup{\'e}r. (4)} \textbf{46} (1), 1--34,
  2013.

\bibitem[dCH85]{decannierehaagerup85}
J. de~Canniere and U. Haagerup.
\newblock {Multipliers of the Fourier algebras of some simple Lie Groups and
  their discrete subgroups.}
\newblock {\em Amer. Journ. Math.} \textbf{107}, 455--500, 1985.

\bibitem[DFSW13]{dawsfimaskalskiwhite13}
M. Daws, P. Fima, A. Skalski, and S. White.
\newblock {The Haagerup property for locally compact quantum groups}.
\newblock To appear in J. Reine Angew. Math., 2013.

\bibitem[Dix57]{dixmier57}
J. Dixmier.
\newblock {\em {Les alg{\`e}bres d'op{\'e}rateurs dans l'espace hilbertien
  (Alg{\`e}bres de von Neumann).}}, volume Fascicule XXV of {\em Cahiers
  scientifiques}.
\newblock Paris: Gauthier-Villars, 1957.

\bibitem[Dri87]{drinfeld86}
V.~G. Drinfel'd.
\newblock {Quantum groups.}
\newblock In {\em {Proceedings of the international congress of mathematicians
  (ICM), Berkeley, USA, August 3--11, 1986}}, volume~II, pages 798 --820.
  Providence, R.I.: American Mathematical Society, 1987.

\bibitem[Gos09]{goswami09}
D. Goswami.
\newblock {Quantum group of isometries in classical and noncommutative
  geometry.}
\newblock {\em Commun. Math. Phys.} \textbf{285} (1), 141--160, 2009.

\bibitem[Haa79]{haagerup79}
U. Haagerup.
\newblock {An example of a nonnuclear C$^*$-algebra, which has the metric
  approximation property.}
\newblock {\em Invent. Math.} \textbf{50} (3), 279–293, 1979.

\bibitem[Haa86]{haagerup86}
U. Haagerup.
\newblock {C$^*$-algebras without the completely bounded approximation
  property}.
\newblock Unpublished manuscript, 1986.

\bibitem[HK94]{haagerupkraus94}
U. Haagerup and J. Kraus.
\newblock {Approximation properties for group C$^*$-algebras and group von
  Neumann algebras.}
\newblock {\em Trans. Amer. Math. Soc.} \textbf{344} (2), 667--699, 1994.

\bibitem[Hou10]{houdayer10-non-free-gropup}
C. Houdayer.
\newblock {Strongly solid group factors which are not interpolated free group
  factors.}
\newblock {\em Math. Ann.} \textbf{346} (4), 969--989, 2010.

\bibitem[Jim85]{jimbo85}
M. Jimbo.
\newblock {A q-difference analogue of $U({\frak g})$ and the Yang-Baxter
  equation.}
\newblock {\em Lett. Math. Phys.} \textbf{10}, 63--69, 1985.

\bibitem[Jol02]{jolissaint02}
P. Jolissaint.
\newblock {Haagerup approximation property for finite von Neumann algebras.}
\newblock {\em J. Operator Theory} \textbf{48} (3), 549–571, 2002.

\bibitem[KR97]{krausruan97}
J. Kraus and Z.-J. Ruan.
\newblock {Multipliers of Kac algebras.}
\newblock {\em Internat. J. Math.} \textbf{8}, 213–248, 1997.

\bibitem[KS97]{klimykschmudgen97}
A. Klimyk and K. Schm{\"u}dgen.
\newblock {\em {Quantum groups and their representations.}}
\newblock Texts and Monographs in Physics. {Berlin: Springer-Verlag}, 1997.

\bibitem[KS09]{koestlerspeicher09-de-finetti}
C. K{\"o}stler and R. Speicher.
\newblock {A noncommutative de Finetti theorem: invariance under quantum
  permutations is equivalent to freeness with amalgamation.}
\newblock {\em Commun. Math. Phys.} \textbf{291} (2), 473--490, 2009.

\bibitem[Maj90]{majid90}
S. Majid.
\newblock {Physics for algebraists: noncommutative and noncocommutative Hopf
  algebras by a bicrossproduct construction.}
\newblock {\em J. Algebra} \textbf{130} (1), 17–64, 1990.

\bibitem[Maj91]{majid91}
S. Majid.
\newblock {Hopf-von Neumann algebra bicrossproducts, Kac algebra
  bicrossproducts, and the classical Yang-Baxter equations.}
\newblock {\em J. Funct. Anal.} \textbf{95} (2), 291–319, 1991.

\bibitem[Neu37]{neumann37}
B.~H. Neumann.
\newblock {Identical relations in groups. I.}
\newblock {\em Math. Ann.} \textbf{114}, 506--525, 1937.

\bibitem[Neu67]{neumann67}
H. Neumann.
\newblock {\em {Varieties of groups}}, volume~37 of {\em Ergebnisse der
  Mathematik und ihrer Grenzgebiete}.
\newblock {Berlin-Heidelberg-New York: Springer-Verlag}, 1967.

\bibitem[Ols70]{olshanskii70}
A.~J. Olshanksii.
\newblock {On the problem of a finite basis of identities in groups}.
\newblock {\em Math. USSR - Izv.} \textbf{4}, 1970.

\bibitem[OP10]{ozawapopa10-cartan1}
N. Ozawa and S. Popa.
\newblock {On a class of II$_1$ factors with at most one Cartan subalgebra.}
\newblock {\em Ann. Math. (2)} \textbf{172} (1), 713--749, 2010.

\bibitem[Ros90]{rosso90-version-duke}
M. Rosso.
\newblock {Alg{\`e}bres enveloppantes quantifi\'ees, groupes quantiques
  compacts de matrices et calcul diff\'erentiel non commutatif}.
\newblock {\em Duke Math. J.} \textbf{61} (1), 11--40, 1990.

\bibitem[Rua96]{ruan96}
Z.-J. Ruan.
\newblock {Amenability of Hopf von Neumann algebras and Kac algebras.}
\newblock {\em J. Funct. Anal.} \textbf{139} (2), 466–499, 1996.

\bibitem[RW12]{raumweber12}
S. Raum and M. Weber.
\newblock {A connection between easy quantum groups, varieties of groups and
  reflection groups}.
\newblock arXiv:1212:4742, 2012.

\bibitem[RW13a]{raumweber13-combinatorial-approach}
S. Raum and M. Weber.
\newblock {The combinatorics of an algebraic class of easy quantum groups.}
\newblock arXiv:1312.1497, 2013.

\bibitem[RW13b]{raumweber13-complete-classification-2}
S. Raum and M. Weber.
\newblock {The full classification of orthogonal easy quantum groups.}
\newblock arXiv:1312.3857, 2013.

\bibitem[RX06]{ricardxu06}
{\'E}. Ricard and Q. Xu.
\newblock {Khintchine type inequalities for reduced free products and
  applications.}
\newblock {\em J. Reine Angew. Math.} \textbf{599}, 27–59, 2006.

\bibitem[Tim08]{timmermann08}
T. Timmermann.
\newblock {\em {An invitation to quantum groups and duality. From Hopf algebras
  to multiplicative unitaries and beyond.}}
\newblock EMS Textbooks in Mathematics. Z{\"u}rich: European Mathematical
  Society, 2008.

\bibitem[Tom06]{tomatsu06}
R. Tomatsu.
\newblock {Amenable discrete quantum groups.}
\newblock {\em J. Math. Soc. Japan} \textbf{58} (4), 949–964, 2006.

\bibitem[vN29]{vonneumann29}
J. von Neumann.
\newblock {Zur allgemeinen Theorie des Masses.}
\newblock {\em Fund. Math.} \textbf{13} (1), 73--111, 1929.

\bibitem[VV03]{vaesvainerman03}
S. Vaes and L. Vainerman.
\newblock {Extensions of locally compact quantum groups and the bicrossed
  product construction.}
\newblock {\em Adv. Math.} \textbf{175} (1), 1--101, 2003.

\bibitem[Wan95]{wang95}
S. Wang.
\newblock {Free products of compact quantum groups.}
\newblock {\em Commun. Math. Phys.} \textbf{167} (3), 671--692, 1995.

\bibitem[Wan98]{wang98}
S. Wang.
\newblock {Quantum symmetry groups of finite spaces.}
\newblock {\em Commun. Math. Phys.} \textbf{195} (1), 195--211, 1998.

\bibitem[Web13]{weber13-classification}
M. Weber.
\newblock {On the classification of easy quantum groups.}
\newblock {\em Adv. Math.} \textbf{245}, 500–533, 2013.

\bibitem[Wor87]{woronowicz87}
S.~L. Woronowicz.
\newblock {Compact matrix pseudogroups.}
\newblock {\em Commun. Math. Phys.} \textbf{111}, 613--665, 1987.

\bibitem[Wor88]{woronowicz88}
S.~L. Woronowicz.
\newblock {Tannaka-Krein duality for compact matrix pseudogroups. Twisted SU(N)
  groups.}
\newblock {\em Invent. Math.} \textbf{93} (1), 35--76, 1988.

\bibitem[Wor91]{woronowicz91}
S.~L. Woronowicz.
\newblock {A remark on compact matrix quantum groups.}
\newblock {\em Lett. Math. Phys.} \textbf{21} (1), 35--39, 1991.

\bibitem[Wor98]{woronowicz98}
S.~L. Woronowicz.
\newblock {Compact quantum groups.}
\newblock In A. Connes et~al., editors, {\em {Quantum symmetries/ Sym{\'e}tries
  quantiques. Proceedings of the Les Houches summer school, Session LXIV, Les
  Houches, France, August 1 -- Septembter 8, 1995}}, pages 845--884. Amsterdam:
  North-Holland, 1998.

\end{thebibliography}

{\small \parbox[t]{200pt}{Sven Raum\\ ENS Lyon \\
    UMPA UMR 5669 \\ F-69364 Lyon cedex 7 \\ France
   \\ {\footnotesize sven.raum@ens-lyon.fr}}
\hspace{15pt}
\parbox[t]{200pt}{Moritz Weber\\ Saarland University, Fachbereich Mathematik\\
   Postfach 151150\\ 66041 Saarbr{\"u}cken \\ Germany
   \\ {\footnotesize weber@math.uni-sb.de}}

\end{document}